\documentclass[final]{siamltex}

\usepackage{pdfsync}

\usepackage[dvips]{color}
\usepackage{graphicx} 

\usepackage{amsmath} 
\usepackage{amssymb}  

\newtheorem{Satz}{Theorem} 
\newtheorem{Lemm}{Lemma} 
\newtheorem{Def}{Definition}
\newtheorem{example}{Example}
\newtheorem{remark}[Lemm]{Remark}

\newcommand{\Simp}{\ensuremath{\mathcal{S}}}

\newcommand \Id  {\operatorname{Id}}

\newcommand \eps   {\varepsilon}

\newcommand{\lel}{\left\langle}
\newcommand{\rir}{\right\rangle}

\newcommand \Iff   {\Leftrightarrow}

\newcommand \N   {\mathbb{N}}
\newcommand \R   {\mathbb{R}}

\newcommand \K   {\mathcal{K}}
\newcommand \Kinf{\mathcal{K_\infty}}
\newcommand \KL  {\mathcal{KL}}
\newcommand \LL  {\mathcal{L}}

\newcommand \PD   {\mathcal{P}}

\newcommand \SSet   {\mathcal{S}}

\newcommand \cK      {\cal{K}}

\title{Input-to-state stability of nonlinear impulsive systems
\thanks{This research is funded by the German Research Foundation (DFG) as a part of Collaborative Research Centre 637 "Autonomous Cooperating Logistic Processes - A Paradigm Shift and its Limitations".}
}

\author{Sergey Dashkovskiy
\thanks{S. Dashkovskiy is with
Erfurt University of Applied Sciences, Altonaer Stra\ss e 25, 99085 Erfurt, Germany
({\tt\small sergey.dashkovskiy@fh-erfurt.de})}%
\and Andrii Mironchenko
\thanks{A. Mironchenko is with Institute of Mathematics, University of W\"urzburg, 
Emil-Fischer-Stra\ss e 40, 97074 W\"urzburg, Germany.
Corresponding author, 
{\tt\small andrii.mironchenko@mathematik.uni-wuerzburg.de}.
}
}

\begin{document}

\maketitle

\begin{abstract}
We prove that impulsive systems, which possess an ISS Lyapunov function, are ISS for time sequences satisfying the fixed dwell-time condition.
If an ISS Lyapunov function is the exponential one, we provide a stronger result, which guarantees uniform ISS of the whole system over sequences satisfying the generalized average dwell-time condition.
Then we prove two small-gain theorems that provide a construction of an ISS Lyapunov function for an interconnection of impulsive systems, if the ISS-Lyapunov functions for subsystems are known. The construction of local ISS Lyapunov functions via linearization method is provided. Relations between small-gain and dwell-time conditions as well as between different types of dwell-time conditions are also investigated.
Although our results are novel already in the context of finite-dimensional systems, we prove them for systems based on differential equations in Banach spaces that makes obtained results considerably more general.
\end{abstract}



\pagestyle{myheadings}
\thispagestyle{plain}

\section{Introduction}

Often in the modeling of real phenomena one has to consider systems, which exhibit both continuous and discontinuous behavior.
A general framework for modeling of such phenomena is a hybrid systems theory
\cite{HCN06}, \cite{GST12}.
Impulsive systems are hybrid systems which state can jump only at moments of time, which are given in advance and do not depend on the state of the system.

The first monograph devoted entirely to impulsive systems is \cite{SaP95}. Recent developments in this field one can find, in particular, in \cite{HCN06}, \cite{Sta09}.

We are interested in stability of impulsive systems with respect to external inputs.  The central concept in this theory is the notion of input-to-state stability (ISS), introduced by E. Sontag in \cite{Son89}, for survey see \cite{Son08}, \cite{DES11}.

Input-to-state stability of impulsive systems has been investigated in recent papers
\cite{HLT08} (finite-dimensional systems)
and \cite{ChZ09}, \cite{LLX11}, \cite{YSH11} (time-delay systems). In \cite{DKM11} ISS of interconnected impulsive systems with and without time-delays has been investigated.
If both continuous and discontinuous dynamics of the system (taken separately from each other) are ISS, then the resulting dynamics of a finite-dimensional impulsive system is also ISS for all impulse time sequences (it is even strongly uniformly ISS, see \cite[Theorem~2]{HLT08}).

More interesting is a study of systems for which either continuous or discrete dynamics is not ISS.
In this case input-to-state stability of an impulsive system cannot be achieved for all sequences of impulse times and one has to introduce restrictions on the class of impulse time sequences to assure ISS.
These conditions are called dwell-time conditions.

The study of ISS of finite-dimensional impulsive systems has been performed in \cite{HLT08}, where it was proved that impulsive systems possessing an exponential ISS-Lyapunov function are uniformly ISS over impulse time sequences satisfying so-called average dwell-time (ADT) condition.
In \cite{ChZ09} a sufficient condition in terms of exponential Lyapunov-Razumikhin functions is provided, which ensures the uniform ISS of impulsive time-delay systems over impulse time sequences satisfying fixed dwell-time (FDT) condition.

In the current literature only exponential ISS Lyapunov functions (or exponential ISS Lyapunov-Razumikhin functions, exponential ISS Lyapunov-Krasovskii functionals) have been exploited for analysis of ISS of impulsive systems.
This restrains the class of systems, which can be investigated by Lyapunov methods, since to our knowledge it is not proved (or disproved) that an exponential ISS-Lyapunov function always exists and even if it does exist, the construction of such function may be a rather sophisticated task.

Another restrictions arise in the study of interconnections of ISS impulsive systems via small-gain theorems (see \cite{JMW96} and \cite{DRW10}). Even if ISS-Lyapunov functions for all subsystems are exponential, an ISS Lyapunov function of the interconnection may be non-exponential, if the gains are nonlinear. Hence many problems can be hardly treated with existing tools. 

Our aim is to develop tools for analysis of ISS of impulsive systems for a broader class of problems. The paper consists of two main parts: in the first part we develop novel Lyapunov-type sufficient conditions for ISS of single impulsive systems and in the second part we consider interconnections of input-to-state stable impulsive systems.

We start by proving that existence of an ISS Lyapunov function (not necessarily exponential) for an impulsive system implies its input-to-state stability over impulsive sequences satisfying nonlinear fixed dwell-time (FDT) condition.
Under slightly weaker FDT condition the uniform global stability of the system over corresponding class of impulse time sequences can be proved.

Furthermore, for the case, when an impulsive system possesses an exponential Lyapunov function, we
generalize the result from \cite{HLT08}, by introducing the generalized average dwell-time (gADT) condition and proving, that an impulsive system, which possesses an exponential ISS Lyapunov function is uniformly ISS over the class of impulse time sequences, which satisfy the gADT condition. We argue, that gADT condition provides in certain sense tight estimates of the class of impulsive time sequences, for which the system is ISS.

Next we develop methods for construction of global as well as local ISS-Lyapunov functions for impulsive systems.

In Section \ref{ISS_Kopplungen} we prove a small-gain theorem for interconnections of impulsive systems, analogous to corresponding theorem for infinite-dimensional systems with continuous behavior \cite{DaM12b}.
Next we show, that if all subsystems possess exponential ISS Lyapunov functions, and the gains are power functions, then the exponential ISS Lyapunov function for the whole system can be constructed.
This result generalizes Theorem 4.2 from \cite{DKM11}, where this statement for linear gains has been proved.

Also we show, how the exponential local ISS-Lyapunov functions for certain classes of control systems can be constructed via linearization method.

Besides these questions we investigate the relations between different types of dwell-time conditions in Section \ref{Versch_KGB}.  
In Section \ref{DT_SGC_Relation} we discuss an interplay between small-gain and dwell-time conditions in the case of interconnected systems, which appears at the stage of selection of gains.

Although our results are novel already in the context of finite-dimensional systems, we prove them for the case of abstract systems, based on equations in Banach spaces. This makes the results more general. The framework we adopt from the paper \cite{DaM12b}, where ISS of infinite-dimensional systems without impulses has been investigated.

The structure of the paper is as follows.
In Section \ref{Prelim} we provide notation and main definitions.
In Section \ref{ISS_Single_Sys} the sufficient conditions for ISS of a single impulsive system via ISS Lyapunov functions are proved.
In Section \ref{Linearisierung_Imp} we show, how a local exponential ISS Lyapunov function can be constructed with the help of linearization. Section \ref{Versch_KGB} is devoted to the relations between different types of dwell-time conditions.
Next we investigate the ISS of interconnected systems via small-gain theorems.
Section \ref{Schluss} concludes the paper.

\section{Preliminaries}
\label{Prelim}

Let $X$ and $U$ denote a state space and a space of input values respectively, and let both of them be Banach.
Take the space of admissible inputs as $U_c:=PC([t_0,\infty),U)$, i.e. the space of piecewise right-continuous functions from $[t_0,\infty)$ to $U$ equipped with the norm
\[
\|u\|_{U_c}:= \sup\limits_{t \geq t_0}\|u(t)\|_U.
\]

Let $T = \{t_1,\ t_2,\ t_3, \ldots \}$ be a strictly increasing sequence of impulse times without finite accumulation points.

Consider a system of the form
\index{system!impulsive}
\begin{equation}
\label{ImpSystem}
\left \{
\begin {array} {l}
{ \dot{x}(t)=Ax(t) + f(x(t),u(t)),\quad t \in [t_0,\infty)  \backslash T,} \\
{ x(t)=g(x^-(t),u^-(t)),\quad t \in T,}
\end {array}
\right.
\end {equation}
where $x(t) \in X$, $u(t) \in U$, $A$ is an infinitesimal generator of a $C_0$-semigroup on $X$ and $f,g:X \times U \to X$.

Equations \eqref{ImpSystem} together with the sequence of impulse times $T$ define an impulsive system.
The first equation of \eqref{ImpSystem} describes the continuous dynamics of the system, and the second describes the jumps of the state at impulse times.

We assume that for each initial condition a solution of the problem \eqref{ImpSystem} exists and is unique.
Note that from the continuity assumptions on the input $u$ it follows that $x(t)$ is piecewise-continuous, and $x^-(t) = \lim\limits_{s \to t-0}x(s)$ exists for all $t \geq t_0$.

For a given set of impulse times by $\phi(t,t_0,x,u)$ we denote the state of \eqref{ImpSystem} corresponding to the initial value $x \in X$, the initial time $t_0$ and to the input $u \in U_c$ at time $t \geq t_0$.

Note that the system \eqref{ImpSystem} is not time-invariant, that is, $\phi(t_2,t_1,x,u) = \phi(t_2+s,t_1+s,x,u)$ doesn't hold for all $\phi_0 \in X$, $u \in U_c$, $t_2 \geq t_1$ and all $s \geq -t_1$.

However, it holds
\begin{eqnarray}
\label{ZeitInvGleichung_Imp}
\phi(t_2,t_1,x,u) = \phi_s(t_2+s,t_1+s,x,u),
\end{eqnarray}
where $\phi_s$ is a trajectory corresponding to the system \eqref{ImpSystem} with impulse time sequence $T_s:= \{t_1+s,\ t_2+s,\ t_3+s, \ldots \}$.

This means that the trajectory of the system \eqref{ImpSystem} with initial time $t_0$ and impulse time sequence $T$ is equal to the trajectory of \eqref{ImpSystem} with initial time $0$ and impulse time sequence $T_{-t_0}$.
Therefore we will assume that $t_0$ is some fixed moment of time and will investigate the stability properties of the system \eqref{ImpSystem} w.r.t. this initial time.

We assume throughout this paper that $x \equiv 0$ is an equilibrium of the unforced system \eqref{ImpSystem}, that is $f(0,0)=g(0,0)=0$.

\begin{Def}
For the formulation of stability properties the following classes of functions are useful:
\begin{equation}\notag
\begin{array}{ll}
{\PD} &:= \left\{\gamma:\R_+\rightarrow\R_+\left|\ \gamma\mbox{ is continuous, }
\right. \gamma(0)=0,\ \gamma(r)>0, r>0 \right\} \\
{\K} &:= \left\{\gamma \in \PD \left|\ \gamma \mbox{ is strictly increasing}  \right. \right\}\\
{\K_{\infty}}&:=\left\{\gamma\in\K\left|\ \gamma\mbox{ is unbounded}\right.\right\}\\
{\LL}&:=\left\{\gamma:\R_+\rightarrow\R_+\left|\ \gamma\mbox{ is continuous and strictly}\right. \text{decreasing with } \lim\limits_{t\rightarrow\infty}\gamma(t)=0\right\}\\
{\KL} &:= \left\{\beta:\R_+\times\R_+\rightarrow\R_+\left|\ \beta \mbox{ is continuous,}\right. \beta(\cdot,t)\in{\K},\ \forall t \geq 0,\  \beta(r,\cdot)\in {\LL},\ \forall r >0\right\}
\end{array}
\end{equation}
Functions of class $\PD$ are called positive definite functions.
\end{Def}

Let us introduce the stability properties for system \eqref{ImpSystem} which we deal with.
\begin{Def}
\label{ISS_One_System}
\index{LISS}
\index{ISS}
For a given sequence $T$ of impulse times we call a system \eqref{ImpSystem} \textit{locally input-to-state stable (LISS)} if there exist $\rho>0$ and $\beta\in\KL,\ \gamma\in\K_{\infty}$, such that $\forall x \in X:\ \|x\|_X~\leq~\rho$, $\forall u \in U_c:\ \|u\|_{U_c} \leq \rho$, $\forall t\geq t_0$ it holds
\begin{equation}
\label{ISS_ImpSys}
\|\phi(t,t_0,x,u)\|_X \leq \beta(\|x\|_X,t-t_0) + \gamma(\|u\|_{U_c}).
\end{equation}
System \eqref{ImpSystem} is \textit{input-to-state stable (ISS)}, if \eqref{ISS_ImpSys} holds for all $x \in X$, $u \in U_c$.

System \eqref{ImpSystem} is called \textit{uniformly ISS} over a given set $\Simp$ of admissible sequences of impulse times if it is ISS for every sequence in $\Simp$, with $\beta$ and $\gamma$ independent of the choice of the sequence from the class $\Simp$.
\end{Def}

\begin{Def}
\label{GS_One_System}
\index{GS}
For a given sequence $T$ of impulse times we call system \eqref{ImpSystem} \textit{globally stable (GS)} if there exist $\xi,\ \gamma~\in~\Kinf$, such that $\forall x \in X$, $\forall u \in U_c$, $\forall t\geq t_0$ it holds
\begin{equation}
\label{GS_ImpSys}
\|\phi(t,t_0,x,u)\|_X \leq \xi(\|x\|_X) + \gamma(\|u\|_{U_c}).
\end{equation}
The impulsive system \eqref{ImpSystem} is \textit{uniformly GS} over a given set $\Simp$ of admissible sequences of impulse times if \eqref{GS_ImpSys} holds for every sequence in $\Simp$, with $\beta$ and $\gamma$ independent of the choice of the sequence.
\end{Def}

%

In the next section we are going to find certain sufficient conditions for an impulsive system of the form \eqref{ImpSystem} to be ISS.

\section[ISS of a single impulsive system]{Lyapunov ISS theory for an impulsive system}
\label{ISS_Single_Sys}

For analysis of (L)ISS of impulsive systems we exploit (L)ISS-Lyapunov functions.
\begin{Def}
\index{ISS-Lyapunov function!impulsive systems}
\label{ISS_Imp_LF}
A continuous function $V:D \to \R_+$, $D \subset X$, $0 \in int(D)$
is called a \textit{LISS-Lyapunov function} for \eqref{ImpSystem} if $\exists \ \psi_1,\psi_2 \in \Kinf$, such that
\begin{equation}
\label{LF_ErsteEigenschaft}
\psi_1(\|x\|_X) \leq V(x) \leq \psi_2(\|x\|_X), \quad x \in D
\end{equation}
holds and $\exists \rho>0,\ \chi \in \Kinf$, $\alpha \in \PD$ and continuous function $\varphi:\R_+ \to \R$ with $\varphi(x)=0$ $\Iff$ $x=0$,  such that
$\forall x \in X:\ \|x\|_X \leq \rho$, $\forall \xi \in U:\ \|\xi\|_U \leq \rho$
it holds
\begin{equation}
\label{ISS_LF_Imp}
V(x)\geq\chi(\|\xi\|_U)\Rightarrow
\left \{
\begin {array} {l}
{ \dot{V}_u(x) \leq - \varphi(V(x))} \\
{ V(g(x,\xi))\leq  \alpha(V(x)),}
\end {array}
\right.
\end{equation}
for all $u \in U_c$, $\|u\|_{U_c} \leq \rho$ and $u(0)=\xi$.
For a given input value $u \in {U_c}$ the Lie derivative $\dot{V}_u(x)$ is defined by
\begin{equation}
\label{LyapAbleitung_2}
\dot{V}_u(x)=\mathop{\overline{\lim}} \limits_{t \rightarrow +0} {\frac{1}{t}(V(\phi_c(t,0,x,u))-V(x)) },
\end{equation}
where $\phi_c$ is a transition map, corresponding to continuous part of the system \eqref{ImpSystem}, i.e. $\phi_c(t,0,x,u)$ is a state of the system \eqref{ImpSystem} at time $t$, if the state at time $t_0:=0$ was $x$, input $u$ was applied and $T=\emptyset$.

If $D=X$ and \eqref{ISS_LF_Imp} holds for all $x \in X$ and $\xi \in U$, then $V$ is called \textit{ISS-Lyapunov function}.
If in addition
\begin{equation}
\label{Lyap_ExpFunk}
\varphi(s) = c s \mbox{ and } \alpha(s) = e^{-d} s
\end{equation}
for some $c,d \in \R$, then $V$ is called \textit{exponential ISS-Lyapunov function with rate coefficients} $c,d$.
\end{Def}

If both $c$ and $d$ are positive, then $V$ decreases along the
continuous flow as well as at each jump. In this case an impulsive system is ISS w.r.t. to all impulse time sequences. If both $c$ and $d$ are negative, then we cannot guarantee ISS of \eqref{ImpSystem} w.r.t. any impulse time sequence. We are interested in the case of
$cd < 0$, where stability properties depend on $T$. In this case
input-to-state stability can be guaranteed under certain restrictions on $T$.
Intuitively, the increase of either $c$ or $d$ leads to less severe
restrictions on $T$.


\begin{remark}
We would like to emphasize that $\phi_c(\cdot,0,x,u)$ does not depend on $T$. Therefore $\dot{V}_u(x)$ and an ISS Lyapunov function $V$ do not depend on the impulse time sequence. This implies, that the existence of an ISS-Lyapunov function doesn't imply in general the input-to-state stability of an impulsive system and additional restrictions on the set of impulse time sequences have to be imposed.
\end{remark}

Our definition of ISS-Lyapunov function is given in the implication form.
The next proposition shows an equivalent way to introduce ISS Lyapunov function, which is frequently used in the literature on hybrid systems, see e.g. \cite{NeT08}. We will use it for the formulation of the small-gain theorem in Section~\ref{ISS_Kopplungen}.
It is a counterpart of \cite[Proposition 2.2.19]{Kos11} where an analogous result for hybrid systems has been shown.

Recall that a function $g: X \times U \to X$ is called locally bounded, if for each $\rho>0$ there exists $K=K(\rho)>0$, so that $\sup_{x \in X: \|x\|_X \leq \rho,\ u \in U: \|u\|_U \leq \rho} \|g(x,u)\|_X \leq K$.

\begin{proposition}
\label{Prop:LyapEquiv}
Let for a continuous function $V:X \to \R_+$ there exist $\psi_1,\psi_2 \in \Kinf$, such that \eqref{LF_ErsteEigenschaft} holds
and $\exists \gamma \in \Kinf$, $\alpha \in \PD$ and continuous function $\varphi:\R_+ \to \R$, $\varphi(0)=0$  such that
for all $\xi \in U$ and all $u \in U_c$ with $u(0)=\xi$ it holds
\begin{equation}
\label{ISS_LF_Imp_2} V(x)\geq\gamma(\|\xi\|_U) \quad \Rightarrow \quad
\dot{V}_u(x) \leq - \varphi(V(x))
\end{equation}
and $\forall x \in X, \xi \in U$ it holds
\begin{equation}
\label{ISS_LF_Imp_3}
V(g(x,\xi))\leq  \max\{ \alpha(V(x)), \gamma(\|\xi\|_U)\}.
\end{equation}
Then $V$ is an ISS Lyapunov function.
If $g$ is locally bounded, then also the converse implication holds.
\end{proposition}

\begin{proof}
"$\Rightarrow$" Pick any $\rho \in \Kinf$ such that
$\alpha(r)<\rho(r)$ for all $r>0$. Then for all $x \in X$ and $\xi
\in U$ from \eqref{ISS_LF_Imp_3} we have
\begin{equation*}
V(g(x,\xi))\leq  \max\{ \rho(V(x)), \gamma(\|\xi\|_U)\}.
\end{equation*}
Define $\chi:= \max\{ \gamma, \rho^{-1} \circ \gamma\} \in \Kinf$.
For all $x \in X$ and $\xi \in U$ such that $V(x) \geq \chi(\|\xi\|_U)$ it follows $\rho(V(x)) \geq \gamma(\|\xi\|_U)$ and hence
\[
V(g(x,\xi))\leq \rho(V(x)).
\]
Since $\chi(r) \geq \gamma(r)$ for all $r>0$, it is clear, that \eqref{ISS_LF_Imp} holds. Thus, $V$ is an ISS-Lyapunov function.

"$\Leftarrow$" 
Let $g$ be locally bounded and let $V$ be an ISS-Lyapunov function for a system \eqref{ImpSystem}.
Then $\exists \chi \in \K$ and $\alpha \in \PD$ such that for all
$x \in X$ and $\xi \in U$ from $V(x)>\chi(\|\xi\|_U)$ it follows $V(g(x,\xi))\leq\alpha(V(x))$.

Let $V(x) \leq \chi(\|\xi\|_U)$. Then $\|x\|_X \leq \psi_1^{-1} \circ \chi(\|\xi\|_U)$.
Define $S(r):=\{ x \in X: \|x\|_X \leq \psi_1^{-1} \circ \chi(r)\}$ and
$\omega(r):= \sup\limits_{\|\xi\|_U \leq r,\ x \in S(r)} \psi_2(\|g(x,\xi)\|_X)$.
This supremum exists since $g$ is locally bounded.
Clearly, $\omega$ is nondecreasing and $\omega(0)=\psi_2(\|g(0,0)\|_X)=0$.
Pick any $\gamma \in \K$: $\gamma \geq \max\{ \omega,\chi\}$. Then for all $x \in X$ and $\xi \in U$ inequality \eqref{ISS_LF_Imp_3} holds and for all $x:$ $\|x\|_X \geq \gamma(\|\xi\|_U)$ estimate \eqref{ISS_LF_Imp_2} holds.
\hfill  \end{proof}

Similarly one can prove the following proposition (which is not a consequence of Proposition~\ref{Prop:LyapEquiv}):
\begin{proposition}
\label{Prop:ExpLyapEquiv}
Let for a continuous function $V:X \to \R_+$ there exist $\psi_1,\psi_2 \in \Kinf$, such that \eqref{LF_ErsteEigenschaft} holds
and $\exists \gamma \in \Kinf$ and $c,d \in \R$ such that
for all $\xi \in U$ and all $u \in U_c$ with $u(0)=\xi$ it holds
\begin{equation*}
V(x)\geq\gamma(\|\xi\|_U)\quad\Rightarrow\quad\dot{V}_u(x) \leq - c
V(x)
\end{equation*}
and $\forall x \in X, \xi \in U$ it holds
\begin{equation*}
V(g(x,\xi))\leq  \max\{ e^{-d} V(x), \gamma(\|\xi\|_U)\}.
\end{equation*}
Then $V$ is an exponential ISS-Lyapunov function.
If $g$ is locally bounded, then also the converse implication holds.
\end{proposition}

%

Now we provide a combination of dwell-time and Lyapunov-type
conditions that guarantees that system \eqref{ImpSystem} is ISS. In
contrast to continuous systems the existence of an ISS-Lyapunov
function for \eqref{ImpSystem} does not automatically imply ISS of
the system with respect to all impulse time sequences. In order to
find the set of impulse time sequences for which the system is ISS
we use the FDT condition \eqref{Gen_Dwell-Time} from \cite{SaP95},
where it was used to guarantee global asymptotic stability of
finite-dimensional impulsive systems without inputs.

For $\theta>0$ define the set $S_{\theta}:=\{\{t_i\}_1^{\infty}
\subset [t_0,\infty) \ : \ t_{i+1}-t_i \geq \theta,\ \forall i \in
\N \}$, consisting of impulse time sequences with distance between
impulse times not less than $\theta$. \index{dwell-time
condition!nonlinear fixed} \index{nonlinear FDT}
\begin{Satz}
\label{NonLin_DwellTime_Satz} Let $V$ be an ISS-Lyapunov function
for \eqref{ImpSystem} and $\varphi,\alpha$ be as in the Definition
\ref{ISS_Imp_LF} and $\varphi \in \PD$. Let for some
$\theta,\delta>0$ and all $a>0$ it hold
\begin{equation}
\label{Gen_Dwell-Time}
\int_a^{\alpha(a)} {\frac{ds}{\varphi(s)}} \leq \theta - \delta.
\end{equation}
Then \eqref{ImpSystem} is ISS for all impulse time sequences $T \in S_{\theta}$.
\end{Satz}

\begin{proof}
Fix arbitrary $u \in U_c$, $\phi_0 \in X$
and choose the sequence of impulse times $T=\{t_i\}_{i=1}^{\infty}$,
$T \in S_{\theta}$. Our aim is to prove ISS of the system
\eqref{ImpSystem} w.r.t. impulse time sequence $T$ by a direct
construction of the functions $\beta$ and $\gamma$ from Definition~\ref{ISS_One_System}.

For the sake of brevity we denote $x(\cdot)=\phi(\cdot,t_0,\phi_0,u)$ and $y(\cdot):=V(x(\cdot))$.

At first assume that $u \equiv 0$. We are going to bound trajectory from above by a function $\beta \in \KL$.

Since  $u \equiv 0$ the following inequalities hold
\begin{equation}
\label{LyapDiffIneq}
 \dot{y}(t) \leq -\varphi (y(t)), \quad t \notin T,
\end{equation}
\begin{equation}
\label{LyapSprungIneq}
y(t) \leq \alpha(y^-(t)), \quad t \in T.
\end{equation}

Take an arbitrary pair $t_i,\ t_{i+1} \in T$.
There are two possibilities: 
either $y(t)>0$ for all $t \in [t_i,t_{i+1})$ or there exists certain time $\hat{t} \in [t_i,t_{i+1})$: $y(\hat{t})=0$ and, since 
$x=0$ is an equilibrium point of the system \eqref{ImpSystem}, $y(t)=0$ for all $t \geq \hat{t}$.

Let us consider the first case. Integrating \eqref{LyapDiffIneq} we obtain
\begin{equation}
\label{HilfUngl_1} 
\int_{t_i}^{t} \frac{dy(\tau)}{\varphi(y(\tau))} \leq - (t-t_{i}), \quad t \in (t_{i},t_{i+1}).
\end{equation}
Fix any $r>0$ and define
\[
F(q):=\int_{r}^{q} \frac{ds}{\varphi(s)}, \quad \forall q>0.
\]
Note that $F:(0,\infty) \to \R$ is a continuous strictly increasing function. Thus, it is invertible on $(0,\infty)$ and $F^{-1}:\R \to (0,\infty)$ is also an increasing function.

Changing variables in \eqref{HilfUngl_1} (which is possible since $y$ is bijective on $(t_i,t_{i+1})$), we can rewrite \eqref{HilfUngl_1} as
\begin{equation}
\label{HilfUngl_2}
F(y(t)) - F(y(t_i) ) \leq - (t-t_{i}).
\end{equation}
Consequently, for $t \in [t_i,t_{i+1})$ it holds
\begin{equation}
\label{V_Abschaetzung}
y(t) \leq F^{-1}\left( F(y(t_i)) - (t-t_{i}) \right).
\end{equation}
Taking in \eqref{HilfUngl_2} a limit $t \to t_{i+1}$ and recalling that $t_{i+1}-t_i \geq \theta$, we obtain
\begin{equation}
\label{HilfUngl_3}
F(y^-(t_{i+1})) - F(y(t_i)) \leq - \theta.
\end{equation}
Using that $y(t_{i+1}) \leq \alpha(y^-(t_{i+1}))$, we obtain the estimate
\[
F(y(t_{i+1})) - F(y(t_i)) \leq \left( F(\alpha(y^-(t_{i+1})))- F(y^-(t_{i+1})) \right)+ \left( F(y^-(t_{i+1})) - F(y(t_i)) \right).
\]
By \eqref{Gen_Dwell-Time} and \eqref{HilfUngl_3} we obtain
\[
F(y(t_{i+1})) - F(y(t_i)) \leq (\theta - \delta) - \theta = - \delta.
\]
From this inequality we have
\begin{equation}
\label{Lyap_Abschaetzung}
y(t_{i+1}) \leq F^{-1} ( F(y(t_i))- \delta ).
\end{equation}
In particular, $y(t) < y(t_i)$, $t \in (t_i,t_{i+1}]$.

From \eqref{Lyap_Abschaetzung} we obtain 
\begin{equation}
\label{Lyap_Abschaetzung_Allgemein}
y(t_{i+1}) \leq F^{-1} ( F(F^{-1} ( F(y(t_{i-1}))- \delta ))- \delta ) = F^{-1} ( F(y(t_{i-1}))- 2 \delta) \leq F^{-1} ( F(y(t_{1}))- i \delta).
\end{equation}

The estimate \eqref{Lyap_Abschaetzung_Allgemein} is valid for all $i$: $F(y(t_{1}))- i \delta \geq \lim_{q \to +0}F(q)$.
Let us denote the maximum of such $i$ by $\hat{i}$ (we set $\hat{i}:=\infty$ if such maximum doesn't exist).

Let us construct a function $\tilde{\beta}:\R_+ \times \R_+ \to \R_+$, which provides a bound for a function $y$. Define 
\[
\tilde{\beta} (r,t_1-t_0) = \max\{y(t_1),\alpha(y(t_1))\},\quad r \geq 0,
\]
where $y(\cdot)$ is a solution of \eqref{LyapDiffIneq}, corresponding to initial condition $y(t_0)=r$.

For all $1 \leq i \leq \hat{i}$ define 
\[
\tilde{\beta} (r,t_{i+1}-t_0) := F^{-1} ( F(\tilde{\beta} (r,t_1-t_0))- i \delta), \quad r \geq 0.
\]
For any $r > 0$, for all $i \leq \hat{i}$ define $\tilde{\beta} (r,\cdot)$ on $(t_{i-1}-t_{i},t_{i}-t_{i+1})$ as an arbitrary continuous decreasing function, which lies above every solution $y(\cdot)$ of \eqref{LyapDiffIneq} with \eqref{LyapSprungIneq}, corresponding to initial condition $y(t_0)=r$.

If $\hat{i}$ is finite, then define $\tilde{\beta} (r,\cdot)$ on $[t_{\hat{i}}-t_0,\infty)$ as a continuous decreasing to 0 function.

By construction, for all $t$ it holds that
\[
y(t) \leq \tilde{\beta} (y_0,t-t_{0}),
\]
where $\tilde{\beta}:\R_+ \times \R_+ \to \R_+$ is continuous w.r.t. the second argument, $\tilde{\beta}(0,t) \equiv 0$ for all $t \geq 0$ and $\tilde{\beta}(a,\cdot)$ is decreasing for all $a>0$. 
We are going to prove, that for all $r \geq 0$ it holds $\tilde{\beta}(r,t) \to 0$ as soon as $t \to \infty$. If $\hat{i} < \infty$, then it follows from the construction. Thus, let $\hat{i} = \infty$.

To show the above claim it is enough to prove, that for all $r>0$ it holds $z_r(t_i) = \tilde{\beta}(r,t_i-t_0) \to 0$, $i \to \infty$.

Let it be false, then due to monotonicity of $z_r$ for some $r{>}0$ $\exists \lim\limits_{i \to \infty} z_r(t_i){=} b_r {>} 0$.

Define $c:=\min_{b \leq s \leq z_r(0)} \varphi(s)$ and observe by the middle-value theorem that
\begin{eqnarray*}
\delta  \leq  F(z_r(t_{i})) - F(z_r(t_{i+1})) = \int_{z_r(t_{i+1})}^{z_r(t_{i})} \frac{ds}{\varphi(s)}  \leq  \frac{1}{c} (z_r(t_{i}) - z_r(t_{i+1})).
\end{eqnarray*}
Hence for all $i$ it holds
\[
z_r(t_{i}) - z_r(t_{i+1}) \geq c \delta,
\]
and the sequence $z_r(t_{i})$ does not converge to a positive limit. We obtained a contradiction to $b_r>0$, thus $z_r(t_{i}) \to 0$, $i \to \infty$. Thus, $\forall r>0$ $\tilde{\beta}(r,\cdot) \in \LL$.

For all $r,t \geq 0$ define $\beta_1(r,t):=\sup_{0 \leq h \leq r}\tilde{\beta}(h,t)$. Clearly, $\beta_1$ is nondecreasing w.r.t. the first argument and $\beta_1(r,t) \geq \tilde{\beta}(r,t)$ for all $r,t \geq 0$. 

Define now $\beta_2(r,t):= \frac{1}{r} \int_{r}^{2r} \beta_1(s,t) ds + re^{-t}$, $\forall r>0, t \geq 0$. Function $\beta_2 \in \KL$ and $\beta_2(r,t) \geq \beta_1(r,t)$, $\forall r,t \geq 0$. 
Hence if $u \equiv 0$ then it holds that
\[
V(x(t)) \leq \beta_2(V(\phi_0),t-t_0), \quad \forall t \geq 0.
\]
Now let $u$ be an arbitrary admissible input. 
Define
\begin{equation}
\label{I1_Definition}
I_1:=\{x \in X: V(x) \leq \chi(\|u\|_{U_c})\},
\end{equation}
For all $t:\ x(t) \notin I_1$ according to \eqref{ISS_LF_Imp} the estimates \eqref{LyapDiffIneq} and \eqref{LyapSprungIneq} hold and consequently
\[
V(x(t)) \leq \beta_2(V(\phi_0),t-t_0), \quad \forall t:\ x(t) \notin I_1.
\]
Let $t^*:=\inf \{t: x(t) \in I_1\}$.
From \eqref{LF_ErsteEigenschaft} we obtain
\begin{equation}
\label{Beta_Abschaetzung}
\|x(t)\|_X \leq \beta(\|\phi_0\|_X,t-t_0),\quad t \leq t^*,
\end{equation}
where $\beta(r,t)=\psi^{-1}_1( \beta_2(\psi_2(r),t))$.

Now we are going to estimate $\|x(t)\|_X$ for $t>t^*$. At first note that a trajectory can leave $I_1$ only by a jump.
If $\|u\|_{U_c}=0$, then $I_1$ is invariant under continuous dynamics, because $x \equiv 0$ is an equilibrium.
Let $\|u\|_{U_c}~>~0$ and let for some $t>t^*$ we have $x(t) \in \partial I_1$, i.e. $y(t)=\chi(\|u\|_{U_c})$.
Then according to the inequality \eqref{ISS_LF_Imp} it holds $\dot{y}(t) \leq -\varphi(y(t)) < 0$ and thus $y(\cdot)$ cannot leave $I_1$ at time $t$.

Define function $\tilde{\alpha}: \R_+ \to \R_+$ by
\[
\tilde{\alpha}(x):= \max\{ \max_{ 0 \leq s \leq \chi(x) } \alpha(s), \chi(x)\},\quad x \in \R_+.
\]
Also let us introduce the set
\[
I_2:=\{x \in X: V(x) \leq \tilde{\alpha}(\|u\|_{U_c})\} \supseteq I_1.
\]

We are going to prove, that $x(t) \in I_2$ for all $t \geq t^*$.

Now let for some $t_k \in T$, $t_k \geq t^*$ it hold $x(t_k) \notin I_1$ and $x(t) \in I_1$ for some $\eps >0$ and for all $t \in (t_k-\eps,t_k)$.
Then $x(t_k) \in I_2$ by construction of the set $I_2$.

But we have proved, that $y(t) < y(t_k)$ as long as $t>t_k$ and $x(t) \notin I_1$. Consequently, $x(t) \in I_2$ for all $t > t^*$.

Thus, for $t> t^*$ it holds
\[
V(x(t)) \leq \tilde{\alpha}(\|u\|_{U_c})
\]
which implies
\[
\|x(t)\|_X \leq \psi^{-1} (\tilde{\alpha}(\|u\|_{U_c})):= \tilde{\gamma} (\|u\|_{U_c}).
\]
Function $\tilde{\gamma}$ is positive definite and nondecreasing, thus, it may be always majorized by a $\K$-function $\gamma$.
Recalling \eqref{Beta_Abschaetzung} we obtain
\begin{equation}
\label{ISS_Abschaetzung}
\|x(t)\|_X \leq \beta(\|\phi_0\|_X,t-t_0) + \gamma(\|u\|_{U_c}), \quad \forall t \geq t_0.
\end{equation}

\hfill  \end{proof}

\begin{remark}
We haven't proved the uniform ISS of the system \eqref{ImpSystem}
w.r.t. $S_{\theta}$. Although the function $\gamma$ by construction
does not depend on the impulse time sequence $T \in S_{\theta}$, the
function $\beta$ does depend. But pick any periodic impulse
time sequence $T=\{t_1,\ldots,t_n,\ldots\} \in S_{\theta}$, that is
$t_{i+1}=t_i + d$ for some $d>0$. Then from the construction of the
function $\beta$ it is clear that \eqref{ImpSystem} is uniformly ISS
over the class $W=\{T_i, i\geq 1\}$, where
$T_i=\{t_i,\ldots,t_n,\ldots\}$.
\end{remark}

\begin{remark}
If the discrete dynamics does not destabilize the system, i.e.
$\alpha(a) \leq a$ for all $a \neq 0$, then the integral on the
right hand side of \eqref{Gen_Dwell-Time} is non-positive for all $a
\neq 0$, and the dwell-time condition \eqref{Gen_Dwell-Time} is
satisfied for arbitrarily small $\theta>0$, that is the system is ISS
for all impulse time sequences without finite  accumulation points.
\end{remark}

We illustrate the application of our theorem on the following toy example.
\begin{example}
Let $T$ be an impulse time sequence. Consider the system $\Sigma$, defined by
\begin{eqnarray}
\label{Beisp_NichtLin}
\left\{
\begin{array}{rl}
\dot{x} =& -x^3 + u ,\ t \notin T \\
x(t) = & x^-(t)+ (x^-(t))^3 + u^-(t),\ t \in T.
\end{array}
\right.
\end{eqnarray}
Consider a function $V:\R \to \R_+$, defined by $V(x)=|x|$.
We are going to prove, that $V$ is an ISS Lyapunov function of the system \eqref{Beisp_NichtLin}.\\
The Lyapunov gain $\chi$ we choose by $\chi(r)=\left( \frac{r}{a}\right)^{\tfrac{1}{3}}$, $r \in \R_+$, for some $a \in (0,1)$. \\
Condition $|x| \geq \chi(|u|)$ implies
\begin{eqnarray*}
\dot{V}(x) & \leq& -(1-a) (V(x))^3, \\
V(g(x,u)) &\leq & V(x)+(1+a)(V(x))^3.
\end{eqnarray*}
Let us compute the integral on the left hand side of \eqref{Gen_Dwell-Time}:
\begin{eqnarray*}
I(y,a) = \int_y^{y+(1+a)y^3} \frac{dx}{(1-a)x^3} = \frac{1+a}{2(1-a)} \frac{2+(1+a)y^2}{(1+(1+a)y^2)^2} \leq \frac{1+a}{(1-a)}.
\end{eqnarray*}
For every $\eps>0$ there exist $a_{\eps}$ such that $I(y,a_{\eps}) \leq 1 + 2\eps$.

Thus, for arbitrary $\eps>0$ we can choose $\theta:=1+\eps$. Note,
that the smaller $\theta$ we take, the larger is the gain. This
demonstrates the trade-off between the size of gains and the density
of allowable impulse times. This dependence plays an important role
in the application of small-gain theorems. See
Section~\ref{DT_SGC_Relation} for details.
\end{example}

A counterpart of Theorem \ref{NonLin_DwellTime_Satz} can be proved also for the GS property.
\begin{Satz}
\label{GS_Satz}
Let all the assumptions of Theorem~\ref{NonLin_DwellTime_Satz} hold with
$\delta:=0$. Then the system \eqref{ImpSystem} is globally stable uniformly over $S_{\theta}$.
\end{Satz}

\begin{proof}
The proof goes along the lines of the proof of the Theorem \ref{NonLin_DwellTime_Satz} up to the inequality \eqref{Lyap_Abschaetzung}, which holds with $\delta=0$.
Then instead of $\tilde{\beta}$ we introduce $\tilde{\xi} \in \Kinf$ by $\tilde{\xi}(r) = \max \{r,\alpha(r)\}$, and instead of estimate \eqref{Beta_Abschaetzung} we have
\begin{equation}
\label{Xi_Abschaetzung}
\|x(t)\|_X \leq \psi^{-1}_1( \tilde{\xi}(\psi_2(\|\phi_0\|_X))) := \xi(\|\phi_0\|_X).
\end{equation}
Thus, for all $t \geq t_0$ we obtain
\begin{equation}
\label{GS_Abschaetzung}
\|x(t)\|_X \leq \xi(\|\phi_0\|_X) + \gamma(\|u\|_{U_c}),
\end{equation}
Note, that the functions $\xi$ and $\gamma$ do not depend on  $t_0$
and on the sequence of impulse times $T$, which proves uniformity.
\hfill  \end{proof}

Now consider the case, when continuous dynamics destabilizes the system and the discrete one stabilizes it. We only sketch the proofs since they are similar to the proofs of Theorems 
\ref{NonLin_DwellTime_Satz} and \ref{GS_Satz}.

Define $\tilde{S}_{\theta}:=\{  \{t_i\}_1^{\infty} \subset [t_0,\infty) \ : \ t_{i+1}-t_i \leq \theta,\ \forall i \in \N \}$.
\begin{Satz}
\label{NonLin_DwellTime_Satz_2} 
Let $V$ be an ISS-Lyapunov function
for \eqref{ImpSystem} and $\varphi,\alpha$ are as in the Definition
\ref{ISS_Imp_LF} with $-\varphi \in \PD$. Let for some $\theta,
\delta >0$ and all $a>0$ it hold
\begin{equation}
\label{Gen_Dwell-Time_2}
\int_{\alpha(a)}^a {\frac{ds}{ - \varphi(s)}} \geq \theta + \delta.
\end{equation}
Then \eqref{ImpSystem} is ISS w.r.t. every sequence from $\tilde{S}_{\theta}$.
\end{Satz}

\begin{proof}
Fix any $x \in X$, take $u \equiv 0$ and choose any sequence of impulse times $T=\{t_i\}_{i=1}^{\infty}$,
$T \in \tilde{S}_{\theta}$. We denote $x(\cdot)=\phi(\cdot,t_0,\phi_0,u)$ and $y(\cdot):=V(x(\cdot))$.

Since $-\varphi \in \PD$, from the inequality \eqref{LyapDiffIneq} we obtain 
\begin{equation}
\label{Satz2_HilfUngl_1} 
\int_{t_i}^{t} \frac{dy(\tau)}{-\varphi(y(\tau))} \leq (t-t_{i}), \quad t \in (t_{i},t_{i+1}).
\end{equation}
Fix any $r>0$ and define
\begin{equation}
\label{NeuFDef}
F(q):=\int_{r}^{q} \frac{ds}{-\varphi(s)}, \quad \forall q>0.
\end{equation}
Note that a function $F$ defined by \eqref{NeuFDef} differs from function $F$ from the proof of Theorem~\ref{NonLin_DwellTime_Satz} by the sign. 

After computations, similar to those from the proof of Theorem~\ref{NonLin_DwellTime_Satz} we obtain for $i \geq 1$ and arbitrary $t \in [t_i,t_{i+1})$
\begin{equation}
\label{Satz2_y_Absch}
y(t) \leq F^{-1} ( F(y(t_{i})) +(t-t_i)).
\end{equation}
For any $i \geq 1$ we have
\begin{equation}
\label{Satz2_Lyap_Abschaetzung_Allgemein}
y(t_{i+1}) \leq F^{-1} ( F(y(t_{i}))- \delta) \leq  F^{-1} ( F(y(t_{1}))- i \delta).
\end{equation}

Now take arbitrary $u \in U_c$ and define $I_1$ as in \eqref{I1_Definition}. Similarly to the proof of Theorem~\ref{NonLin_DwellTime_Satz} there exist time $t^*$ (depending on $\phi_0$ and $u$) and $\beta \in \KL$ so that 
\begin{equation}
\label{Satz2_Beta_Abschaetzung}
\|x(t)\|_X \leq \beta(\|\phi_0\|_X,t-t_0),\quad t \leq t^*.
\end{equation}
Let us find an estimate of $\|x(t)\|_X$ for $t > t^*$. Since $\alpha < id$, the trajectory of \eqref{ImpSystem} cannot leave the set $I_1$ by a jump. 

Denote $t_s:=\inf\{t_i: t_i>t^*\}$. 
For $t \in [t^*,t_s)$ inequality \eqref{Satz2_y_Absch} implies
\begin{equation*}
V(x(t)) \leq F^{-1} ( F(V(x(t^*))) + \theta) \leq F^{-1} ( F( \chi(\|u\|_{U_c}) ) + \theta).
\end{equation*}
Since $V(x(t_s))<V(x(t^*))$ due to \eqref{Satz2_Lyap_Abschaetzung_Allgemein} we obtain
\begin{equation}
\label{Satz2_AG_Abschaetzung}
\|x(t)\|_X \leq \gamma(\|u\|_{U_c}), \quad t > t^*,
\end{equation}
where $\gamma(r) =\psi_1^{-1} \left( F^{-1} ( F( \chi(r) ) + \theta) \right)$.
This estimate together with \eqref{Satz2_Beta_Abschaetzung} proves ISS of the impulsive system \eqref{ImpSystem} for all impulse time sequences $T \in \tilde{S}_{\theta}$.
\end{proof}

\begin{Satz}
\label{GS_Satz_2}
Let the assumptions of the Theorem \ref{NonLin_DwellTime_Satz_2} hold with
$\delta:=0$. Then the system \eqref{ImpSystem} is GS uniformly over $\tilde{S}_{\theta}$.
\end{Satz}

\begin{proof}
The proof is a combination of ideas from the proofs of Theorems~\ref{NonLin_DwellTime_Satz_2} and \ref{GS_Satz}.
\end{proof}

\subsection{Sufficient condition in terms of exponential ISS-Lyapunov functions}

Theorem \ref{NonLin_DwellTime_Satz} can be used, in particular, for systems possessing exponential ISS-Lyapunov functions, but for this particular class of systems even stronger result can be proved.

For a given sequence of impulse times denote by $N(t,s)$ the number of jumps within the interval $(s,t]$.

\begin{Satz}
\label{ExpCase}
Let $V$ be an exponential ISS-Lyapunov function for \eqref{ImpSystem} with corresponding coefficients $c \in \R$, $d \neq 0$. For arbitrary function $h:\R_+ \to(0,\infty)$, for which there exists $g \in \LL$: $h(x) \leq g(x)$ for all $x \in \R_+$
consider the class $\SSet[h]$ of impulse time-sequences, satisfying the generalized average dwell-time (gADT) condition:
\begin{equation}
\index{dwell-time condition!generalized average}
\index{generalized ADT}
\label{Dwell-Time-Cond}
-dN(t,s) - c(t-s) \leq  \ln h(t-s), \quad \forall t\geq s \geq t_0.
\end{equation}
Then the system \eqref{ImpSystem} is uniformly ISS over $\SSet[h]$.
\end{Satz}

\begin{proof}
Pick any $h$ as in the statement of the theorem.
Fix arbitrary $u \in U_c$, $\phi_0 \in X$, choose the increasing sequence of impulse times $T=\{t_i\}_{i=1}^{\infty} \in \SSet[h]$ and denote $x(t)= \phi(t,t_0,\phi_0,u)$ for short.

Due to the right-continuity of $x(\cdot)$ the interval $[t_0,\infty)$ can be decomposed into subintervals as
$[t_0,\infty) = \cup_{i=0}^{\infty} [t^*_i,t^*_{i+1})$ (the case, when this decomposition is finite, can be treated in the same way),
so that $\forall k \in \N \cup \{0\}$ the following inequalities hold
\begin{equation}
\label{GegebeneSchranke1}
V(x(t)) \geq \chi(\|u\|_{U_c}) \text{ for } t \in [t^*_{2k},t^*_{2k+1}),
\end{equation}
\begin{equation}
\label{GegebeneSchranke2}
V(x(t)) < \chi(\|u\|_{U_c})  \text{ for } t \in [t^*_{2k+1},t^*_{2k+2}).
\end{equation}

Let us estimate $V(x(t))$ on the time-interval $I_{2k}=(t^*_{2k},t^*_{2k+1}]$ for arbitrary $k \in \N \cup \{0\}$.

Within the interval $I_{2k}$ there are $r_k:=N(t^*_{2k},t^*_{2k+1})$ jumps at times $t^{k}_1,\ldots, t^{k}_{r_k}$.
To simplify the notation, we denote also $t^k_0:=t^*_{2k}$.

For $t \in (t^k_i,t^k_{i+1}]$, $i=0,\ldots,r_k$ we have $V(x(t)) \geq \chi(\|u\|_{U_c})$, thus from \eqref{ISS_LF_Imp} and \eqref{Lyap_ExpFunk} we obtain
\begin{equation}
\label{ContDynExpSatz}
 \dot{V}(x(t)) \leq -c V(x(t)), \ t \in (t^k_i,t^k_{i+1}]
\end{equation}
and thus
\[
V(x^-(t^k_{i+1})) \leq e^{-c(t^k_{i+1}-t^k_{i})} V(x(t^k_{i})).
\]
At the impulse time $t=t^k_{i+1}$ we know from \eqref{ISS_LF_Imp} and \eqref{Lyap_ExpFunk} that
\[
 V(x(t^k_{i+1})) \leq e^{-d} V(x^-(t^k_{i+1}))
\]
and consequently
\[
V(x(t^k_{i+1})) \leq e^{-d-c(t^k_{i+1}-t^k_{i})} V(x(t^k_{i})).
\]
For all $t \in I_{2k}$ from \eqref{ContDynExpSatz} and previous inequality we obtain the  following estimate
\[
V(x(t)) \leq e^{-d\cdot N(t,t^*_{2k}) -c(t-t^*_{2k})} V(x(t^*_{2k})).
\]
Dwell-time condition \eqref{Dwell-Time-Cond} implies
\begin{eqnarray}
\label{IntervallSchranke}
V(x(t)) \leq h(t-t^*_{2k}) V(x(t^*_{2k})), t \in I_{2k}.
\end{eqnarray}

Take $\tau:=\inf\{t \geq t_0: V(x(t)) \leq \chi(\|u\|_{U_c})\}$. We are going to find an upper bound of the trajectory on $[t_0,\tau]$ as a $\KL$-function.

Taking in \eqref{IntervallSchranke} $t^*_{2k}:=t_0$ we obtain
\begin{eqnarray}
\label{KL_Schranke}
V(x(t)) \leq  h(t-t_0) V(\phi_0).
\end{eqnarray}
According to assumptions of the theorem, $\exists g \in \LL$: $h(x) \leq g(x)$ for all $x \in \R_+$.
Using \eqref{LF_ErsteEigenschaft}, we obtain that $\forall t \in [t_0,\tau]$ it holds
\begin{eqnarray*}
\|x(t)\|_X \leq \psi_1^{-1}(g(t-t_0) \psi_2(\|\phi_0\|_X))=: \beta(\|\phi_0\|_X,t-t_0).
\end{eqnarray*}

On arbitrary interval of the form $[t^*_{2k+1},t^*_{2k+2})$, $k \in
\N \cup \{0\}$ we have already the bound on $V(x(t))$ by
\eqref{GegebeneSchranke2}. Since $t^*_{2k+2}$ can be an impulse
time, we have the estimate
\[
V(x(t^*_{2k+2})) \leq \max\{1,e^{-d}\} \chi(\|u\|_{U_c}).
\]
From the properties of $h$ it follows, that
 $\exists C_{\lambda}=\sup_{x\geq 0}\{h(x)\} < \infty$.
Hence for arbitrary $t > \tau$
we obtain with the help of \eqref{IntervallSchranke} the estimate
\[
V(x(t)) \leq C_{\lambda} \max\{1,e^{-d}\} \chi(\|u\|_{U_c}).
\]
Overall, for all $t \geq t_0$ we have
\begin{eqnarray*}
\|x(t)\|_X \leq \beta(\|\phi_0\|_X,t-t_0) + \gamma(\|u\|_{U_c}),
\end{eqnarray*}
where
$\gamma(r)=\psi_1^{-1}(C_{\lambda} \max\{1,e^{-d}\} \chi(r))$.
This proves, that the system \eqref{ImpSystem} is ISS. The uniformity is clear since the functions $\beta$ and $\gamma$ do not depend on the impulse time sequence.
\hfill  \end{proof}

\begin{remark}
Theorem \ref{ExpCase} generalizes Theorem 1 from \cite{HLT08}, where
this result for the function $h$ with $h(x)=e^{\mu-\lambda x}$ has
been proved.
\end{remark}

The condition \eqref{Dwell-Time-Cond} is tight, i.e., if for some sequence $T$ the function $N(\cdot,\cdot)$ does not satisfy the condition \eqref{Dwell-Time-Cond} for every function $h$ from the statement of the Theorem \ref{ExpCase}, then one can construct a certain system \eqref{ImpSystem} which will not be ISS w.r.t. the impulse time sequence $T$.

This one can see from the following simple example. Consider
\begin{equation*}
\left \{
\begin {array} {l}
\dot{x}=-cx, \quad t \notin T,\\
x(t)=e^{-d}x^-(t), \quad t \in T
\end {array}
\right.
\end {equation*}
with initial condition $x(0)=x_0$.
Its solution for arbitrary time sequence $T$ is given by
\[
x(t)=e^{-dN(t,t_0)-c(t-t_0)}x_0.
\]
If $T$ does not satisfy the gADT condition, then $e^{-dN(t,t_0)-c(t-t_0)}$ cannot be estimated from above by $\LL$-function, and consequently, the system under consideration is not GAS.

We state also the local version of Theorem \ref{ExpCase}:
\begin{Satz}
\label{ExpCase_Local}
Let $V$ be an exponential LISS-Lyapunov function for \eqref{ImpSystem} with corresponding coefficients $c \in \R$, $d \neq 0$. For arbitrary function $h:\R_+ \to(0,\infty)$, s.t. $\exists g \in \LL$: $h(x) \leq g(x)$ for all $x \in \R_+$
there exist a constant $\rho(h)$, such that the system \eqref{ImpSystem} is uniformly LISS with this $\rho$ over  the class $\SSet[h]$ of impulse time-sequences, satisfying \eqref{Dwell-Time-Cond}.
\end{Satz}

\begin{proof}
The proof of this result is similar to the proof of Theorem \ref{ExpCase}. The only difference is that one has to choose $\rho$ small enough to guarantee that the system evolves on the domain of definition of ISS-Lyapunov function $V$.
\hfill  \end{proof}

\subsection{Relations between different types of dwell-time conditions}
\label{Versch_KGB}

For the system \eqref{ImpSystem} which possesses an exponential ISS-Lyapunov function we have introduced two different types of dwell-time conditions: generalized ADT condtion \eqref{Dwell-Time-Cond} and fixed dwell-time condition \eqref{Gen_Dwell-Time}. In this section we are going to find a relation between these conditions as well as between ADT condition from \cite{HLT08}. See also \cite{Hes04}, where some other sets of switching signals and relations between them have been investigated.

Taking in the gADT \eqref{Dwell-Time-Cond} $h(x) = e^{\mu-\lambda x}$ for some $\mu,\lambda>0$, we obtain the ADT condition from \cite{HeM99b}, \cite{HLT08}:
\begin{equation}
\index{dwell-time condition!average}
\index{ADT}
\label{Classical_Dwell-Time-Cond}
-dN(t,s) - (c - \lambda)(t-s) \leq  \mu, \quad \forall t\geq s \geq t_0.
\end{equation}
The set of impulse time sequences, which satisfies this condition we denote $\SSet[\mu,\lambda]:=\SSet[ e^{\mu-\lambda \cdot }]$.

The gADT condition \eqref{Dwell-Time-Cond} provides for a system \eqref{ImpSystem} in addition to jumps, allowed by ADT \eqref{Classical_Dwell-Time-Cond} the possibility to jump
infinite number of times (on the time-interval of the infinite length), however, these jumps must be "not too close" to each other. Consider, for example $h(x) = (x+1)e^{\mu-\lambda x}$. This choice of $h$ leads to the following dwell-time condition
\begin{equation*}
-dN(t,s) - (c - \lambda)(t-s) \leq  \mu + \ln(t-s+1), \quad \forall t\geq s \geq t_0.
\end{equation*}
Locally (for small $t-s$) it holds $\mu >> \ln(t-s+1)$, and we obtain the estimate similar to \eqref{Classical_Dwell-Time-Cond}. But for large $t-s$ the above DT condition is a considerably weaker restriction than classical ADT condition \eqref{Classical_Dwell-Time-Cond}.

Of course, the more extra jumps we allow, the larger are the gain $\gamma$ and  function
$\beta$, which can be seen from the proof of Theorem~\ref{ExpCase}.

For a given sequence of impulse times denote by $N^*(t,s)$ the number of jumps within the time-interval $[s,t]$. The set of impulse time sequences, for which \eqref{Classical_Dwell-Time-Cond} holds with $N^*(t,s)$ instead of $N(t,s)$, denote by $\SSet^*[\mu,\lambda]$.
We need the following lemma (see \cite[Lemma 3.12.]{DKM11}):
\begin{Lemm}
\label{Dwell-Time_Equiv}
Let $c,d \in \R$, $d \neq 0$ be given. Then $\SSet[\mu,\lambda] = \SSet^*[\mu,\lambda]$ for all $\mu,\lambda>0$.
\end{Lemm}

%
%

Let us show the relation between ADT and FDT conditions.

If the system \eqref{ImpSystem} possesses an exponential ISS Lyapunov function with rate coefficients $c,d \in \R$, $d<0$ then Theorem \ref{NonLin_DwellTime_Satz} guarantees, that for all $\delta>0$ and $\theta >0$, such that
\begin{equation}
\index{dwell-time condition!fixed}
\index{FDT}
\label{FDT_ExpLF_Case}
\int_a^{\alpha(a)} {\frac{ds}{\varphi(s)}} = \frac{-d}{c} \leq \theta - \delta
\end{equation}
holds the system \eqref{ImpSystem} is ISS for the time-sequences from the class $S_{\theta}$.

Clearly, for all positive numbers $\theta$, $\delta$, satisfying \eqref{FDT_ExpLF_Case} there exists $\lambda>0$, such that the following condition holds with the same $\theta$
\begin{equation}
\label{Lin_Dwell-Time}
\frac{1}{\theta} \leq \frac{c - \lambda}{-d},
\end{equation}
and vice versa.

For a given $\lambda$ the smallest $\theta$ (which corresponds to the largest $S_{\theta}$) is given by $\theta_* = \frac{-d}{c - \lambda}$.

Next lemma provides an equivalent representation of the set $S_{\theta_*}$.
\begin{Lemm}
Let $c>0$ and $d<0$ be given. Then it holds $S_{\theta_*} = \SSet[-d,\lambda]$.
\end{Lemm}

\begin{proof}
Clearly, for arbitrary $T \in S_{\theta_*}$ it holds
\begin{equation*}
N^*(t,s) \leq 1 + \frac{c - \lambda}{-d}(t-s), \quad \forall t\geq s \geq t_0,
\end{equation*}
or
\begin{equation}
\label{Dwell-time_SpezFall}
-dN^*(t,s) - (c - \lambda)(t-s)  \leq -d, \quad \forall t\geq s \geq t_0.
\end{equation}
On the contrary, let \eqref{Dwell-time_SpezFall} hold. Then for $t-s = k   \theta_*$ we obtain $N^*(t,s) \leq k+1$ and for $t-s \in ((k-1)   \theta_*,k   \theta_*)$
it follows $N^*(t,s) \leq k$ (since $N^*(t,s)$ is a natural number).
This proves that $S_{\theta_*} = \SSet^*[-d,\lambda]$. From Lemma \ref{Dwell-Time_Equiv} the claim of the lemma follows.
\hfill  \end{proof}

In other words, Theorem \ref{Gen_Dwell-Time}, applied to the exponential ISS Lyapunov functions, states that if the system \eqref{ImpSystem} possesses an exponential ISS Lyapunov function $V$ with rate coefficients $c,d$, then for all $\lambda>0$  the system \eqref{ImpSystem} is ISS for all sequences from the class $\SSet[-d,\lambda]$.

\begin{figure}
\centering
  \setlength{\unitlength}{1bp}%
  \begin{picture}(323.83, 198.07)(0,0)
  \put(0,0){\includegraphics{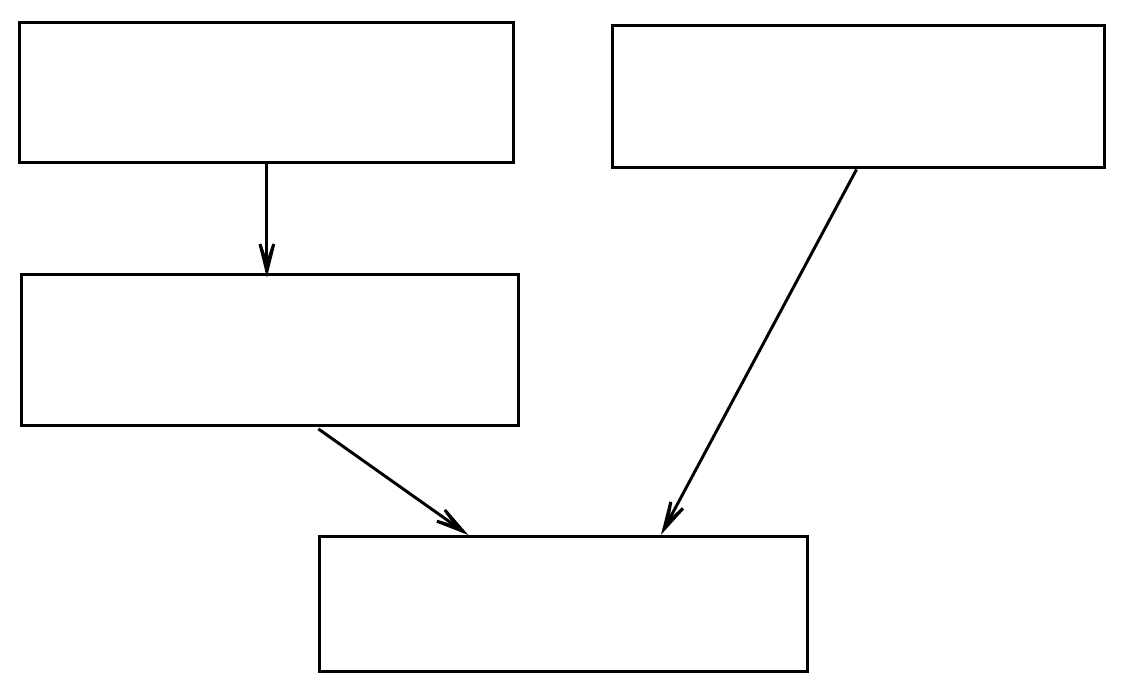}}	
\put(40.81,179.23){\fontsize{9.23}{17.07}\selectfont generalized ADT}
\put(6.81,159.23){\fontsize{9.23}{17.07}\selectfont $-dN(t,s) - c(t-s) \leq  \ln h(t-s)$}
\put(52.81,135.23){\fontsize{9.23}{17.07}\selectfont $h(x):=e^{\mu-\lambda x}$}

\put(52.41,106.68){\fontsize{9.23}{17.07}\selectfont average DT}
  \put(16.41,86.68){\fontsize{9.23}{17.07}\selectfont $-dN(t,s) - (c-\lambda)(t-s) \leq \mu$}
\put(225.38,178.38){\fontsize{9.23}{17.07}\selectfont fixed DT}
  \put(200.38,158.38){\fontsize{11.23}{17.07}\selectfont $\int_a^{\alpha(a)} {\frac{ds}{\varphi(s)}} \leq \theta - \delta$}

    \put(92.81,61.23){\fontsize{9.23}{17.07}\selectfont $\mu:=-d$}
  \put(136.44,21.45){\fontsize{12.23}{17.07}\selectfont $\frac{1}{\theta} \leq \frac{c- \lambda}{-d}$}
  \end{picture}%
\caption{Relations between different types of dwell-time conditions}
\label{DT_Cond_Figure}
\end{figure}
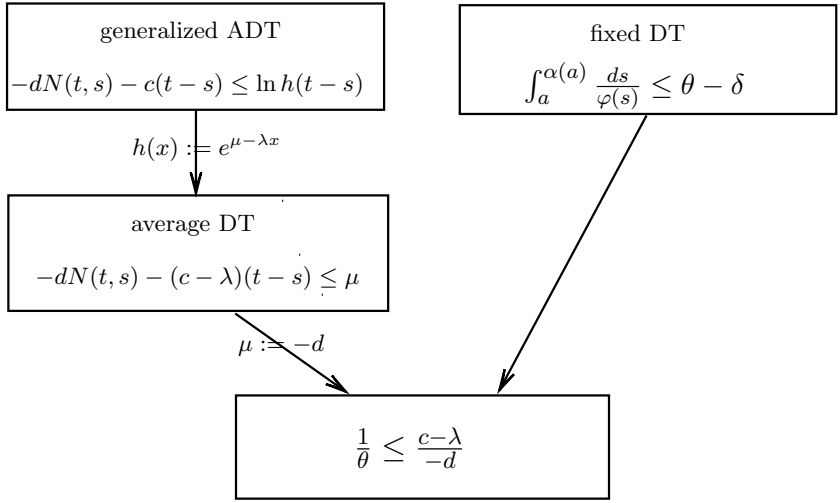

\begin{remark}
Note that for $\mu \in (0,-d)$ the set of the impulse time sequences, which are allowed by ADT condition are $\SSet[\mu,\lambda]=\emptyset$. Indeed, by the ADT condition
for small enough $t-s$ we obtain
\[
N(t,s)-\frac{c-\lambda}{-d}(t-s) \leq \frac{\mu}{-d} < 1,
\]
i.e.  $N(t,s) = 0$. Covering $[0,\infty)$ by small enough intervals, we obtain that $N(t_0,\infty)=0$, and the impulses are not allowed.
\end{remark}

The relations between different types of dwell-time conditions are summarized in Figure~\ref{DT_Cond_Figure}. 

In the next half of the paper we are going to provide the methods for construction of ISS-Lyapunov functions. In the next section we show how the local exponential ISS-Lyapunov functions can be constructed via linearization technique. Afterwards we focus our attention on construction of global ISS-Lyapunov functions for interconnected impulsive systems.

\subsection{Constructions of exponential LISS Lyapunov functions via linearization}
\label{Linearisierung_Imp}

Consider an impulsive system \eqref{ImpSystem} on a Hilbert space $X$ with a scalar product $\lel \cdot, \cdot \rir$, and let $A$ be the infinitesimal generator of an analytic semigroup on $X$ with the domain of definition $D(A)$.
Let a function $f:X \times U \to X$ be defined on some open set $Q$, $(0,0) \in Q$.

In \cite[Theorem 3]{DaM12b} it was proved that for the system \eqref{ImpSystem} with $T = \emptyset$ (when only continuous behavior is allowed) under certain conditions a LISS-Lyapunov function can be constructed.

In this section we prove a counterpart of \cite[Theorem 3]{DaM12b} for impulsive systems, which allows us to construct an exponential LISS-Lyapunov function for linearizable systems of the form \eqref{ImpSystem}.

Let us assume, that $f$ and $g$ can be decomposed in the following way
\[
f(x,u)=Bx+ Cu + f_1(x,u),
\]
\[
g(x,u)=Dx+ Fu + g_1(x,u),
\]
where $C,F \in L(U,X)$, $B,D \in L(X)$.
Here we denote by $L(U,X)$ a space of linear bounded operators from $U$ to $X$, $L(X):=L(X,X)$.

Let also for each constant $w>0$ there exists $\rho>0$, such that $\forall x: \|x\|_X \leq \rho,\ \forall u: \|u\|_U \leq \rho$ it holds
\[
\|f_1(x,u)\|_X \leq w (\|x\|_X+\|u\|_U),
\]
\[
\|g_1(x,u)\|_X \leq w (\|x\|_X+\|u\|_U).
\]

We recall that a self-adjoint operator $P$ on the Hilbert space $X$ is coercive, if $\exists \epsilon >0$, such that $\lel Px,x \rir \geq \epsilon \|x\|^2_X \quad \forall x \in D(P)$.
The largest of such $\epsilon$ is called the lower bound of an operator $P$.

Consider a linear approximation of continuous dynamics of a system \eqref{ImpSystem}:
\begin{equation}
\label{LinearInfiniteSyst_Imp}
\dot{x}=Rx + Cu,
\end{equation}
where $R=A+B$ is the infinitesimal generator of an analytic semigroup (which we denote by $T$), as a sum of the generator of an analytic semigroup $A$ and bounded operator $B$.

%

We have the following theorem:
\begin{Satz}
\index{linearization theorem!impulsive systems}
\label{LinearisationSatz}
If the system \eqref{LinearInfiniteSyst_Imp} is ISS and if there exists a bounded coercive operator $P$, satisfying
\begin{equation*}
\lel Rx,Px \rir + \lel Px,Rx \rir=-\|x\|^2_X, \quad \forall x \in D(A),
\end{equation*}
then a LISS-Lyapunov function of \eqref{ImpSystem} can be constructed in the form
\begin{equation}
\label{LF_LinearInfiniteSyst}
V(x)=\lel Px,x \rir.
\end{equation}
\end{Satz}

\begin{proof}
Since $P$ is bounded and coercive, for some $\epsilon>0$ it holds 
\[
\epsilon \|x\|^2_X \leq \lel Px,x \rir \leq \|P\| \|x\|^2_X, \quad \forall x \in X,
\]
and the estimate \eqref{LF_ErsteEigenschaft} is verified.

Define $\chi \in \Kinf$ by $\chi(r)=\sqrt{r}$, $r \geq 0$.
In \cite[Theorem 3]{DaM12b}
it was proved, that for small enough $\rho_1 >0$, $\forall x: \|x\|_X~\leq~\rho_1,\ \forall u: \|u\|_U~\leq~\rho_1$ it holds
\[
\|x\|_X \geq \chi(\|u\|_U) \quad \Rightarrow \quad \dot{V}(x) \leq - r \|x\|^2_X \leq
- \frac{r}{\|P\|} V(x)
\]
for some $r>0$.

Now we estimate $V(g(x,u))$:
\begin{eqnarray*}
V(g(x,u)) &= &\lel P(Dx+ Fu + g_1(x,u)),Dx+ Fu + g_1(x,u) \rir \\
&\leq& \|P\| \left(\|D\|^2 \|x\|^2_X + \|F\|^2 \|u\|^2_U + 2\|D\| \|F\| \|x\|\|u\|_U  \right.\\
 & &  + 2(\|D\| \|x\|_X + \|F\| \|u\|_U) w (\|x\|_X+\|u\|_U)  +  \left. w^2 (\|x\|_X+\|u\|_U)^2 \right).
\end{eqnarray*}
One can verify, that $\exists r_2, \rho_2 >0$, such that
$\forall x: \|x\|_X \leq \rho_2,\ \forall u: \|u\|_U \leq \rho_2$
\[
\|x\|_X \geq \chi(\|u\|_U) \quad \Rightarrow \quad V(g(x,u)) \leq r_2 \|x\|^2_X \leq \frac{r_2}{\eps} V(x).
\]
Taking $\rho:=\min\{\rho_1,\rho_2\}$, we obtain, that $V$ is an exponential LISS Lyapunov function for a system \eqref{ImpSystem}.
\hfill  \end{proof}

\section{ISS of interconnected impulsive systems}
\label{ISS_Kopplungen}

In the previous subsection we have developed a linearization method for construction of LISS-Lyapunov functions for impulsive systems \eqref{ImpSystem}. Now we are going to provide a method for construction of ISS-Lyapunov functions for interconnected systems which is based on the knowledge of ISS-Lyapunov functions for subsystems.

Let a Banach space $X_i$ be the state space of the $i$-th subsystem, $i=1,\ldots,n$, and $U$ and $U_c=PC(\R_+,U)$ be the space of input values and of input functions respectively.

Define $X=X_{1}\times\ldots\times X_{n}$,
which is a Banach space, which we endow with the norm $\|\cdot\|_{X}:=\|\cdot\|_{X_{1}}+\ldots+\|\cdot\|_{X_{n}}$.

The input space for the $i$-th subsystem is
$\tilde{X}_i:=X_1 \times \ldots \times X_{i-1} \times X_{i+1} \times \ldots \times X_n \times U$.
The norm in $\tilde{X}_i$ is given by
\[
\|\cdot\|_{\tilde{X}_i}:=\|\cdot\|_{X_{1}}+\ldots + \|\cdot\|_{X_{i-1}} + \|\cdot\|_{X_{i+1}} + \ldots +\|\cdot\|_{X_{n}} + \|\cdot\|_{U}.
\]
The elements of $\tilde{X}_i$ we denote by $\tilde{x}_i=(x_1,\ldots,x_{i-1},x_{i+1},\ldots,x_n,\xi) \in \tilde{X}_i$.

Also let $T=\{t_1,\ldots, t_k,\ldots \}$ be a sequence of impulse times for all subsystems (we assume, that all subsystems jump at the same time).

Consider the system consisting of $n$ interconnected impulsive subsystems:
\begin{equation}
\label{Kopplung_N_Systeme_Imp}
\left \{
\begin {array} {l}
{ \dot{x}_i(t)=A_ix_i(t) + f_i(x_1(t),\ldots,x_n(t),u(t)),\quad t \notin T,} \\
{ x_i(t)=g_i(x^-_1(t),\ldots,x^-_n(t), u^-(t)),\quad t \in T,} \\
{i=\overline{1,n}}
\end {array}
\right.
\end {equation}

Here $A_i$ is the generator of a $C_{0}$-semigroup on $X_{i}$, $f_i,g_i:X \times U \to X_i$, and we assume that the solution of each subsystem exists, is unique and forward-complete.


For $x_i \in X_i,\ i=1,\ldots, n$ define
$x=(x_{1},\ldots,x_{n})^{T}$, $f(x,u)=(f_{1}(x,u),\ldots,f_{n}(x,u))^{T}$,
$g(x,u)=(g_{1}(x,u),\ldots,g_{n}(x,u))^{T}$.

By $A$ we denote the diagonal operator $A:=diag(A_{1},\ldots,A_{n})$, i.e.:
\[
A=\left(
\begin{array}{cccc}
A_{1} & 0 & \ldots & 0\\
0 & A_{2} & \ldots & 0\\
\vdots & \vdots & \ddots & \vdots\\
0 & 0 & \ldots & A_{n}
\end{array}\right)
\]
Domain of definition of $A$ is given by $D(A)=D(A_{1})\times\ldots\times D(A_{n})$.
Clearly, $A$ is the generator of a $C_{0}$-semigroup on $X$.

We rewrite the system \eqref{Kopplung_N_Systeme_Imp} in the vector form:
\begin{equation}
\label{KopplungHauptSys_Imp}
\left \{
\begin {array} {l}
{ \dot{x}(t)=Ax(t) + f(x(t),u(t)),\quad t \notin T} \\
{ x(t)=g(x^-(t),u^-(t)),\quad t \in T.}
\end {array}
\right.
\end {equation}

Before analyzing of a system \eqref{KopplungHauptSys_Imp} we note that for general interconnections of impulsive systems each subsystem may possess its own 
sequence of impulse times $T_i$. The interconnected system in this case takes the form
\begin{equation}
\label{Kopplung_N_Systeme_Imp_NFolg}
\left \{
\begin {array} {l}
{ \dot{x}_i(t)=A_ix_i(t) + f_i(x_1(t),\ldots,x_n(t),u(t)),\quad t \notin T_i,} \\
{ x_i(t)=g_i(x^-_1(t),\ldots,x^-_n(t), u^-(t)),\quad t \in T_i,} \\
{i=\overline{1,n}}
\end {array}
\right.
\end{equation}

In contrast to the system \eqref{Kopplung_N_Systeme_Imp}, it is impossible to rewrite the system \eqref{Kopplung_N_Systeme_Imp_NFolg} in the form \eqref{ImpSystem}.
One can construct the aggregated sequence of impulse times for the whole system as $T:=\cup_{i=1}^nT_i$, but the function $g$ for the whole system will still depend on the time-sequences $T_i$, $i=1,\ldots,n$. This means that \eqref{Kopplung_N_Systeme_Imp_NFolg} is not simply a "large scale impulsive system", but a more complicated type of hybrid systems. The theory developed in this paper as well as (to our knowledge) in other current literature on ISS of impulsive systems cannot be applied to such systems. 
Development of such theory is an interesting topic for future research.

Let us proceed with analysis of a system \eqref{KopplungHauptSys_Imp}.

According to the Proposition~\ref{Prop:LyapEquiv} for the $i$-th subsystem of a system \eqref{Kopplung_N_Systeme_Imp} the definition of an ISS-Lyapunov function can be written as follows.
A continuous function $V_{i}:X_{i}\to\R_{+}$ is an ISS-Lyapunov function
for $i$-th subsystem of \eqref{Kopplung_N_Systeme_Imp}, if three properties hold (to avoid unnecessary complications we consider only the case when the continuous dynamics of the subsystems is stabilizing):

\begin{enumerate}
    \item There exist functions $\psi_{i1},\psi_{i2}\in\Kinf$, such that:
\[
\psi_{i1}(\|x_{i}\|_{X_{i}})\leq V_{i}(x_{i})\leq\psi_{i2}(\|x_{i}\|_{X_{i}}),\quad\forall x_{i}\in X_{i}
\]
\item  There exist $\chi_{ij},\chi_{i}\in\K$, $j=1,\ldots,n$, $\chi_{ii}:=0$ and
$\varphi_i \in \PD$, so that for all $x_i \in X_i$, for all $\tilde{x}_i \in \tilde{X}_i$ and for all $v \in PC(\R_+,\tilde{X}_i)$ with $v(0)=\tilde{x}_i$ from
\begin{equation}
\label{ISS-LF_Annahme}
V_i(x_{i})\geq\max\{ \max_{j=1}^{n}\chi_{ij}(V_{j}(x_{j})),\chi_{i}(\|\xi\|_{U})\},
\end{equation}
it follows
\begin{equation}
\label{GainImplikationNSyst_Imp}
\dot{V_{i}}(x_i(t)) \leq -\varphi_i\left(V_i(x_i(t))\right),
\end{equation}
where
\[
\dot{V}_i(x_{i})=\mathop{\overline{\lim}}\limits _{t\rightarrow+0}\frac{1}{t}(V_i(\phi_{i,c}(t,0,x_{i},v)))-V_i(x_{i})),
\]
and $\phi_{i,c}:\R_+ \times \R_+ \times X_i \times PC(\R_+,\tilde{X}_i) \to X_i$ is the solution (transition map) of the $i$-th subsystem of \eqref{Kopplung_N_Systeme_Imp} for the case if $T=\emptyset$.
\item There exists $\alpha_{i} \in \PD$, such that for gains defined above and for all $x \in X$ and for all $\xi \in U$ it holds
\begin{equation}
\label{GainImplikationNSyst_Imp_Diskt}
V_i(g_{i}(x,\xi))\leq \max\{ \alpha_i(V_i(x_i)),\max_{j=1}^{n}\chi_{ij}(V_{j}(x_{j})),\chi_{i}(\|\xi\|_{U})\}.
\end{equation}
\end{enumerate}

If $\varphi_i(y)=c_i y$ and $\alpha_i (y) =e^{-d_i}y$ for all $y \in \R_+$, then $V_i$ is called an exponential ISS-Lyapunov function for the $i$-th subsystem of \eqref{Kopplung_N_Systeme_Imp} with rate coefficients $c_i, d_i \in \R$.

Lyapunov gains $\chi_{ij}$ characterize the interconnection structure of
subsystems.
Let us introduce the gain operator $\Gamma:\R_{+}^{n}\rightarrow\R_{+}^{n}$
defined by
\begin{equation}
\label{operator_gamma}
\Gamma(s):=\left(\max_{j=1}^{n}\chi_{1j}(s_{j}),\ldots,\max_{j=1}^{n}\chi_{nj}(s_{j})\right),\ s\in\R_{+}^{n}.
\end{equation}

We recall the notion of $\Omega$-path (see \cite{DRW10,Rue10}), useful for investigation of stability of interconnected systems and for a construction of a Lyapunov function of the whole system.
\begin{Def}
\label{Omega-Path-Def}
A function $\sigma=(\sigma_{1},\dots,\sigma_{n})^{T}:\R_{+}^{n}\rightarrow\R_{+}^{n}$,
where $\sigma_{i}\in\K_{\infty}$, $i=1,\ldots,n$ is called an \textit{$\Omega$-path},
if it possesses the following properties:
\begin{enumerate}
    \item $\sigma_{i}^{-1}$ is locally Lipschitz continuous on $(0,\infty)$;
    \item for every compact set $P\subset(0,\infty)$ there are finite
constants $0<K_{1}<K_{2}$ such that for all points of differentiability
of $\sigma_{i}^{-1}$ we have
\begin{align*}
0<K_{1}\leq(\sigma_{i}^{-1})'(r)\leq K_{2},\quad\forall r\in P ;
\end{align*}
\item
\begin{align}
\label{sigma cond 2}
\Gamma(\sigma(r))<\sigma(r),\ \forall r>0.
\end{align}
\end{enumerate}
\end{Def}

If operator $\Gamma$ satisfies the small-gain condition
\begin{align}
\label{smallgaincondition}
\Gamma(s)\not\geq s,\ \forall\ s\in\R_{+}^{n}\backslash\left\{ 0\right\},
\end{align}
then $\Omega$-path exists \cite{DRW10}.

Now we prove a small-gain theorem for nonlinear impulsive systems. The technique for treatment of the discrete dynamics is adopted from \cite{NeT08} and \cite{DaK09}.
\begin{Satz}
\label{KleinGainSatz_ImpSys}
\index{small-gain theorem!in Lyapunov formulation!for impulsive systems}
Consider the system \eqref{Kopplung_N_Systeme_Imp}.
Let $V_{i}$ be the ISS-Lyapunov function for $i$-th subsystem of \eqref{Kopplung_N_Systeme_Imp} with corresponding gains $\chi_{ij}$.
If the corresponding operator $\Gamma$ defined by \eqref{operator_gamma}
satisfies the small-gain condition \eqref{smallgaincondition}, then
an ISS-Lyapunov function $V$ for the whole system can be constructed as 
\begin{align}
\label{NeuLyapFun}
V(x):=\max_{i}\{\sigma_{i}^{-1}(V_{i}(x_{i}))\},
\end{align}
where $\sigma=(\sigma_{1},\ldots,\sigma_{n})^{T}$ is an $\Omega$-path.
The Lyapunov gain of the whole system can be chosen as 
\begin{equation}
\label{ChiDef}
\chi(r):=\max_{i}\sigma_{i}^{-1}(\chi_{i}(r)).
\end{equation}
\end{Satz}

\begin{proof}
The part of the proof related to continuous behavior is identical to the proof of \cite[Theorem 5]{DaM12b}.
There it was proved, that $\forall x \in X,\ \xi \in U$ from $V(x)\geq\chi(\|\xi\|_U)$ it follows
\begin{eqnarray*}
\frac{d}{dt}V(x) \leq  - \varphi(V(x)),
\end{eqnarray*}
for
\begin{eqnarray}
\label{Varphi_Def}
\varphi(r):=\min_{i=1}^{n} \left\{\left(\sigma_{i}^{-1}\right)^{\prime}(\sigma_{i}(r))\varphi_{i}(\sigma_{i}(r)) \right\}.
\end{eqnarray}
Function $\varphi$ is positive definite, because $\sigma_{i}^{-1} \in \Kinf$ and all $\varphi_i$ are positive definite functions.

Thus, implication \eqref{ISS_LF_Imp_2} is verified and it remains to check \eqref{ISS_LF_Imp_3}
(the estimation of ISS-Lyapunov function on the jumps).
With the help of inequality \eqref{GainImplikationNSyst_Imp_Diskt} we make for all $x \in X$ and $\xi \in U$ the following estimates
\begin{eqnarray*}
\label{Estimate_ImpTime_Algem}
V(g(x,\xi)) & = &\max_{i}\{\sigma_{i}^{-1}(V_{i}(g_{i}(x,\xi)))\}   \\
 &\leq& \max_{i}\{\sigma_{i}^{-1}\left( \max\{ \alpha_i(V_i(x_i)),\max_{j=1}^{n}\chi_{ij}(V_{j}(x_{j})),\chi_{i}(\|\xi\|_{U})\} \right)\}  \\
&=&
\max\{ \max_{i}\{\sigma_{i}^{-1} \circ \alpha_i(V_i(x_i)) \},
\max_{i,j \neq i} \{ \sigma_{i}^{-1} \circ \chi_{ij}(V_{j}(x_{j}))\},
\max_{i}\{\sigma_{i}^{-1} \circ \chi_{i}(\|\xi\|_{U}) \}  \} \\
&=&
\max\{ \max_{i}\{\sigma_{i}^{-1} \circ \alpha_i \circ \sigma_{i} \circ \sigma_{i}^{-1}(V_i(x_i))\},
\max_{i,j \neq i} \{ \sigma_{i}^{-1} \circ \chi_{ij} \circ \sigma_{j} \circ \sigma_{j}^{-1} (V_{j}(x_{j}))\},  \\
& &
\phantom{\max\{ } \max_{i}\{\sigma_{i}^{-1} \circ \chi_{i}(\|\xi\|_{U}) \}  \}.  \\
\end{eqnarray*}
Define $\tilde{\alpha}:=\max_{i} \{ \sigma_{i}^{-1} \circ \alpha_i \circ \sigma_{i} \}$. Since $\alpha_i \in \PD$, then $\tilde{\alpha} \in \PD$. Pick any $\alpha^*\in \K$: $\alpha^*(r) \geq \tilde{\alpha}(r)$, $r \geq 0$. Then the following estimate holds
\[
\max_{i}\{\sigma_{i}^{-1} \circ \alpha_i \circ \sigma_{i} \circ \sigma_{i}^{-1}(V_i(x_i))\} \leq \alpha^*(\max_{i}\{\sigma_{i}^{-1}(V_i(x_i))\}) = \alpha^*(V(x)).
\]
Define also $\eta:=\max_{i,j \neq i} \{ \sigma_{i}^{-1} \circ \chi_{ij} \circ \sigma_{j} \}$ and note that according to \eqref{sigma cond 2}
\[
\eta =\max_{i,j \neq i} \{ \sigma_{i}^{-1} \circ \chi_{ij} \circ \sigma_{j} \} <
\max_{i,j \neq i} \{ \sigma_{i}^{-1} \circ \sigma_{i} \} = id.
\]

We continue estimates of $V(g(x,\xi))$:
\[
V(g(x,\xi)) \leq
\max\{\alpha^*(V(x)), \eta(V(x)),\chi(\|\xi\|_{U})  \} = \max\{\alpha(V(x)),\chi(\|\xi\|_{U})  \},
\]
where
\begin{equation}
\label{Alpha_Def}
\alpha:=\max\{\alpha^*,\eta\}.
\end{equation}
According to Proposition~\ref{Prop:LyapEquiv} the function $V$ is an ISS-Lyapunov function of the system \eqref{ImpSystem}.
\hfill  \end{proof}

\begin{remark}
Our small-gain theorem has been formulated for Lyapunov functions in the form used in Proposition~\ref{Prop:LyapEquiv}. According to the Proposition~\ref{Prop:LyapEquiv} this formulation can be transformed to the standard formulation, and from the proof it is clear, that the functions $\alpha$ and $\varphi$ remain the same after the transformation.
Next in order to check, whether the system \eqref{KopplungHauptSys_Imp} is ISS, one should use Theorem~\ref{NonLin_DwellTime_Satz}.
\end{remark}

\subsection{Small-gain theorem for exponential ISS-Lyapunov functions}

If an exponential ISS-Lyapunov function for a system \eqref{ImpSystem} is given, then Theorem \ref{ExpCase} provides us with the tight estimations of the set of impulse time sequences, w.r.t. which the system \eqref{ImpSystem} is ISS and hence the exponential ISS-Lyapunov functions are "more valuable", than the general ones.

We may hope, that if ISS-Lyapunov functions for all subsystems of \eqref{Kopplung_N_Systeme_Imp} are \textit{exponential}, then the expression \eqref{NeuLyapFun} at least for certain type of gains provides  the \textit{exponential} ISS-Lyapunov function for the whole system.
In this subsection we are going to prove the small-gain theorem of this type.

Firstly note the following fact
\begin{proposition}
Let operator $\Gamma$ satisfy the small-gain condition \eqref{smallgaincondition}.
Then for arbitrary $a \in int(\R^n_+)$ the function
\begin{equation}
\label{OmegaPfad}
\sigma(t)=Q(at), \forall t \geq 0
\end{equation}
satisfies
\begin{align}
\Gamma(\sigma(r)) \leq \sigma(r),\ \forall r>0.
\end{align}
Here $Q:\R^n_+ \to \R^n_+$ is defined by
\[
Q(x):=MAX\{x,\Gamma(x), \Gamma^2(x),\ldots, \Gamma^{n-1}(x)\},
\]
with $\Gamma^n(x)=\Gamma \circ \Gamma^{n-1}(x)$, for all $n \geq 2$.
The function $MAX$ for all $h_i \in \R^n$, $i=1,\ldots,m$ is defined by
\[
z=MAX\{h_1,\ldots,h_m\} \in \R^n,\quad z_i:=\max\{h_{1i},\ldots,h_{mi}\}.
\]
\end{proposition}

\begin{proof}
The result follows from \cite[Proposition 2.7 and Remark 2.8]{KaJ11}
\hfill
\end{proof}


Define the following class of functions
\[
P:=\{f:\R_+ \to \R_+:\exists a \geq 0,\ b>0:\;  f(s)=a s^b \; \forall s \in \R_+\}.
\]


\begin{Satz}
\label{KleinGain_Sonderfall}
\index{small-gain theorem!in Lyapunov formulation!for impulsive systems with exponential LF}
Let $V_i$ be an eISS Lyapunov function for the $i$-th subsystem of \eqref{Kopplung_N_Systeme_Imp} with corresponding gains $\chi_{ij}$, $i=1,\ldots,n$. Let also $\chi_{ij} \in P$ and let  the small-gain condition \eqref{smallgaincondition} hold. Then the
function $V:X \to \R_+$, defined by \eqref{NeuLyapFun}, where the $\sigma$ is given by \eqref{OmegaPfad}, is an eISS Lyapunov function for the whole system \eqref{KopplungHauptSys_Imp}.
\end{Satz}

\begin{proof}
Take the $\Omega$-path $\sigma$ as in \eqref{OmegaPfad}. It satisfies all the conditions of an $\Omega$-path, see Definition~\ref{Omega-Path-Def}, but with $\leq$ instead of $<$ in \eqref{sigma cond 2}. However, the proof of Theorem~\ref{KleinGainSatz_ImpSys} is true also for such "quasi"-$\Omega$-path.

According to Theorem~\ref{KleinGainSatz_ImpSys} function $V$, defined by \eqref{NeuLyapFun} is an ISS Lyapunov function. We have only to prove, that it is an exponential one.

For all $f,g \in P$ it follows $f \circ g \in P$, thus for all $i$ it holds that $\sigma_i(t)=\max\{f^i_1(t),\ldots,f^i_{r_i}(t)\}$, where all $f^i_k \in P$ and $r_i$ is finite.

Thus, for each $i$ there exists a partition of $\R_+$ into sets $S^i_{j}$, $j=1,\ldots,k_i$  (i.e. $\cup_{j=1}^{k_i} S^i_{j} = \R_+$ and $S^i_{j} \cap S^i_{s} = \emptyset$, if
$j \neq s$), such that $\sigma_i^{-1}(t) = a_{ij}t^{p_{ij}}$ for some $p_{ij}>0$ and all $t \in S^i_j$.  This partition is always finite, because all $f^i_j \in P$, and two such functions intersect in no more than one point, distinct from zero.


Thus, for all $i\in \{1,\ldots,n\}$ define a set
\[
M_{i}=\left\{ x\in X:\sigma_{i}^{-1}(V_{i}(x_{i})) > \sigma_{j}^{-1}(V_{j}(x_{j})),\,\,\forall j=1,\ldots,n,\ j\neq i\right\}.
\]
Let $x \in M_i$ and $V_i(x_i) \in S^i_j$. Then the condition \eqref{smallgaincondition} implies
(see the proof of \cite[Theorem 5]{DaM12b})
\begin{eqnarray*}
\frac{d}{dt}V(x)& = &\frac{d}{dt} ( \sigma_{i}^{-1}(V_i(x_i)) ) = \frac{d}{ds} ( a_{ij}s^{p_{ij}}) (V_i(x_i)) \frac{d}{dt} (V_i(x_i))
\end{eqnarray*}
Now using \eqref{GainImplikationNSyst_Imp} and \eqref{Lyap_ExpFunk} we have
\begin{eqnarray*}
\frac{d}{dt}V(x)    &\leq& - c_i a_{ij} p_{ij} (V_i(x_i))^{p_{ij}} \leq -c  V(x),
\end{eqnarray*}
where $c=\min_{i,j}\{ c_i p_{ij} \}$.

We have to prove, that the function $\alpha$ from \eqref{Alpha_Def} can be estimated from above by linear function.
We choose
$\alpha^*:=\tilde{\alpha}=\max_{i} \{ \sigma_{i}^{-1} \circ \alpha_i \circ \sigma_{i} \}$.

For any fixed $t\geq 0$ it holds that $\sigma_{i}^{-1} \circ \alpha_i \circ \sigma_{i}(t) =const=:c_i$ since $\alpha_i$ are linear and $\sigma_{i}^{-1}$ are piecewise power functions.
This implies that for some constant $k$ it holds that $\alpha^*(t) \leq  k t$ for all $t \geq 0$.

Since function $\eta$ from the proof of Theorem~\ref{KleinGainSatz_ImpSys} satisfies $\eta < id$, it is clear that one can take $\alpha:=\max\{k,1\}\Id$, and the theorem is proved.
\end{proof}

\begin{remark}
The obtained exponential ISS-Lyapunov function can be transformed to the implication form with the help of Proposition~\ref{Prop:ExpLyapEquiv}. Then Theorem~\ref{ExpCase} can be used in order to verify ISS of the system \eqref{KopplungHauptSys_Imp}.
\end{remark}


%
Let us demonstrate how one can analyze stability of interconnected impulsive systems on a simple example.
Let $T = \{t_k\}$ be a sequence of impulse times. Consider two interconnected nonlinear impulsive systems
\begin{align*}
\dot{x}_1(t) =&\ -x_1(t) + x_2^2(t), \ t \notin T,\\
x_1(t)=&\ e^{-1} x_1^-(t),  \ t \in T
\end{align*}
and
\begin{align*}
\dot{x}_2(t) =&\ -  x_2(t) + 3 \sqrt{|x_1(t)|},  \ t \notin T,\\
x_2(t)=&\  e^{-1} x_2^-(t),  \ t \in T.
\end{align*}

Both subsystems are uniformly ISS (even strongly uniformly ISS, see \cite{HLT08}) for all impulse time sequences, since continuous and discrete dynamics stabilize the subsystems and one can easily construct exponential ISS Lyapunov functions (with certain Lyapunov gains) with positive rate coefficients for both subsystems.
However, any admissible Lyapunov gains, corresponding to such ISS-Lyapunov functions will not satisfy small-gain condition, since the continuous dynamics of the interconnected system is not stable.
Therefore in order to find the classes of impulse time sequences for which the interconnected system is GAS, we have to seek for ISS-Lyapunov functions (and corresponding Lyapunov gains) with one negative rate coefficient.

Take the following exponential ISS-Lyapunov functions and Lyapunov gains for subsystems
\begin{align*}
V_1(x_1)=|x_1|, \quad   \gamma_{12}(r)=\tfrac{1}{a} r^2,\\
V_2(x_2)=|x_2|, \quad   \gamma_{21}(r)=\tfrac{1}{b} \sqrt{r},
\end{align*}
where $a,b>0$. We have the following implications
\begin{align*}
|x_1| \geq \gamma_{12} (|x_2|) \Rightarrow \dot{V}_1(x_1) \leq  (a-1) V_1(x_1),\\
|x_2| \geq \gamma_{21} (|x_1|) \Rightarrow \dot{V}_2(x_2) \leq  (3b-1) V_2(x_2).
\end{align*}
The small-gain condition
\begin{align}
\label{KGB_NichtLin}
\gamma_{12} \circ \gamma_{21} (r) = \frac{1}{ab^2}  r < r, \quad \forall r>0
\end{align}
is satisfied, if it holds
\begin{align}
\label{KGB_UnserSyst}
h(a,b):=ab^2 >1.
\end{align}
Take an arbitrary constant $s$ such that $\frac{1}{b} < \frac{1}{s} < \sqrt{a}$. Then $\Omega$-path can be chosen as
\[
\sigma_1(r)=r,\quad  \sigma_2(r)=\frac{1}{s} \sqrt{r},\ \forall r \geq 0.
 \]
Then
\[
\sigma_2^{-1}(r)=s^2 r^2,\ \forall r \geq 0.
\]
In this case an ISS-Lyapunov function for the interconnection, constructed by small-gain design, is given by
\begin{align*}
V(x)=\max\{ |x_1|, s^2 |x_2|^2  \} , \quad \text{where } \frac{1}{b} < \frac{1}{s} < \sqrt{a} \text{ and } x=(x_1,x_2)^T
\end{align*}
and we have the estimate
 \begin{align}
\label{LyapAbsch_Diskr}
V(g(x))=V(e^{-1} x) \leq e^{-1} V(x).
\end{align}
Thus, we can take $d=-1$ for the interconnection. The estimates of the continuous dynamics for $V$ are as follows:
For $|x_1| \geq s^2 x_2^2>\frac{1}{a}x_2^2=\gamma_{12}(|x_2|)$ it holds
\[
\frac{d}{dt}V(x)= \frac{d}{dt} |x_1| \leq (a-1)  |x_1| = (a-1)V(x),
\]
and $|x_1| \leq s^2 x_2^2 < \gamma_{21}^{-1} (|x_2|) $ implies
\[
\frac{d}{dt}V(x)= \frac{d}{dt} \left(s^2 x_2^2 \right) = \frac{d}{dt} \left(s^2 V_2(x_2)^2\right) \leq    2(3b-1) s^2 |x_2|^2 = 2(3b-1)V(x).
\]
Overall, for all $x$ we have:
\begin{align}
\label{LyapAbsch_Cont}
\frac{d}{dt} V(x) \leq \max\{(a-1), 2(3b-1) \} V(x).
\end{align}

Function $h$, defined by \eqref{KGB_UnserSyst}, is increasing w.r.t. both arguments (since $a,b >0$), hence in order to minimize $c:= \max\{(a-1), 2(3b-1) \}$, we have to choose $(a-1)= 2(3b-1)$. Then, from \eqref{KGB_NichtLin} we obtain the inequality
\begin{align*}
(1+2(3b-1)) b^2 > 1.
\end{align*}

Thus, the best choice for $b$ is $b \approx 0.612$ and $V$ is an exponential ISS-Lyapunov function for an interconnection with rate coefficients with $d=-1$ and $c=2 \cdot(3 \cdot 0.612 - 1) = 1.672$.

The ISS-Lyapunov function for an interconnection is constructed, and one can apply Theorem \ref{ExpCase} in order to obtain the classes of impulse time sequences for which the interconnection is GAS.
%
%

\subsection{Relation between small-gain and dwell-time conditions}
\label{DT_SGC_Relation}

So far we have seen how small-gain and dwell-time conditions can be
used to verify stability of interconnected systems. The small gain
condition \eqref{smallgaincondition} requires that the gains of subsystems must be small
enough so that their cycle compositions are less then the identity, namely
\begin{eqnarray}
\label{ZyklenBedin}
\gamma_{k_1k_2} \circ \gamma_{k_2k_3} \circ \ldots \circ \gamma_{k_{p-1}k_p} (s) < s
\end{eqnarray}
for all $(k_1,...,k_p) \in \{1,...,n\}^p$, where $k_1=k_p$ and for all $s>0$.
The condition in cyclic form \eqref{ZyklenBedin} is equivalent to the condition \eqref{smallgaincondition}, see \cite{DRW07}, and is widely used in the literature \cite{JiW08}.

In particular a large gain of one subsystem can
be compensated by a small gain of another one to satisfy
\eqref{smallgaincondition}. A choice of gains depends on the choice
of an ISS-Lyapunov function in its turn.

The dwell-time condition is imposed on $\alpha$ and $\varphi$ from
\eqref{ISS_LF_Imp} or the rate coefficients $c$ and $d$ in case of
exponential ISS-Lyapunov functions. It requires that the jumps happen
with a certain frequency.

The inequalities \eqref{ISS_LF_Imp} show how fast the value of
$V(x(\cdot))$ changes outside of the region $\{x: V(x) <
\gamma(|u|)\}$ with the time $t$. In the previous example we have
seen that the larger is the gain function, the larger the rate
coefficients $c$ and $d$ can be chosen and hence the more impulse
time sequences satisfy the dwell-time condition
\eqref{Dwell-Time-Cond}. However in case of interconnected systems
large gains may lead to the situation, where the small-gain condition is not satisfied.
Hence there is a trade-off between the size of the gains (which we
want to have as small as possible) and the decay rate of
$V(x(\cdot))$. This leads to interdependence in the choice of gains
and rate coefficients in the stability analysis of interconnected
systems. In general case this dependence is rather involved. To shed
light on this issue we restrict ourselves in this section to the case of
systems possessing exponential ISS-Lyapunov functions with linear
gains.

Consider an interconnected impulsive system of the form
\eqref{Kopplung_N_Systeme_Imp}, and assume that for each $i$ there
is a positive definite and radially unbounded continuous
function $V_i$ for the $i$-th subsystem, such that for almost all
$x_i \in X_i$ and all $u \in U$ the following
dissipative inequalities hold:
\begin{equation}
\label{Dissip_Cont_Inequality}
\dot{V}_i(x_i) \leq -\tilde{c} V_i(x_i) + \max_{j \neq i}\{ {\chi}_{ij} V_j(x_j),\chi_i(\|u\|_U)\},
\end{equation}
\begin{equation}
\label{Dissip_Jumps_Inequality}
V_i(g_i(x,u)) \leq \max\{ e^{- {d}} V_i(x_i),\chi_i(\|u\|_U)\},
\end{equation}
where $\chi_{ij} \in \R_+$, $\tilde{c},  {d} \in \R$, and
$\chi_i\in{\cK}$ can be nonlinear functions. We have assumed here
for simplicity, that the subsystem affect each other during
continuous flow only. At the impulse times the jumps of subsystems
are independent on each other.

Let us illustrate the trade-off mentioned above. By the inequalities
\eqref{Dissip_Cont_Inequality} and \eqref{Dissip_Jumps_Inequality}
function $V_i$ is an ISS-Lyapunov function in dissipative
formulation for the $i$-th subsystem, see \cite{HLT08}. This form
provides us with a freedom to choose the gains during transformation
of equation \eqref{Dissip_Cont_Inequality} from the dissipation form
into the implication form which we need in order to apply
Theorem~\ref{KleinGain_Sonderfall}.

Let $k \in (0,\infty)$ be the scaling coefficient that allows to
adjust the gains to satisfy the small-gain condition. We define
\begin{equation}
\label{GainsDef} \gamma_{ij}:= \frac{1}{k} \chi_{ij}, \quad
\gamma_i:= \frac{1}{k} \chi_i, \quad
\Gamma_k:=(\gamma_{ij})_{i,j=1,\ldots,n}.
\end{equation}
 If
\begin{align*}
\max_{j \neq i}\{ {\gamma}_{ij} V_j(x_j),\gamma_i(\|u\|_U)\} =
\frac{1}{k} \max_{j \neq i}\{ {\chi}_{ij} V_j(x_j),\chi_i(\|u\|_U)\}
\leq V_i(x_i)
\end{align*}
holds, then it follows from \eqref{Dissip_Cont_Inequality} that
\begin{align}
\label{Optimal_Estimate} \dot{V}_i(x_i) \leq  \left(- \tilde{c}+ k
\right) V_i(x_i):= - c_{k} V_i(x_i),
\end{align}
holds for almost all $x_i$, with $c_k:=\tilde{c}-k$.

This shows that $V_i$ is an exponential ISS-Lyapunov function of the
$i$-th subsystem in the sense of Definition~\ref{ISS_Imp_LF} with
the rate coefficients $c_k$ and $d$ and gains $\gamma_{ij}$ for
which our small-gain theorem can be applied.


Define the linear operator $\Gamma_k:\R^n_+ \to \R^n_+$ by
$(\Gamma_k(s))_i =  \max\limits_{j \neq i} \{\gamma_{ij} s_j\}$.
For this operator the small-gain condition \eqref{smallgaincondition} is equivalent (see \cite{DRW06b})  to
\[
\rho(\Gamma_k)<1 \quad\Leftrightarrow\quad \rho:=\rho\big( (\chi_{ij})_{i,j=1}^n\big)<k,
\]
where $\rho(\cdot)$ denotes the spectral radius of a matrix.

In this case according to Theorem~\ref{KleinGain_Sonderfall} an
exponential Lyapunov function can be constructed, moreover, an
$\Omega$-path can be chosen as a vector of linear functions and the
rate coefficients of the ISS-Lyapunov function for a whole system
will be $c_k$ and $d$.

If $k \in (\rho, \tilde{c})$ and $d>0$, then both rate coefficients
of the exponential ISS-Lyapunov functions $V_i$ are positive and
hence the system under consideration is ISS for all impulsive time
sequences.

Let us consider the case when $d<0$ and $k \in (\rho, \tilde{c})$,
when the rate coefficients are of different signs, and consequently
one has to use dwell-time conditions in order to find the classes of
impulse time sequences w.r.t. which the system is ISS.

The dwell-time condition \eqref{Classical_Dwell-Time-Cond} for $d<0$
reads in this situation as
\begin{equation}
\label{Dwell-Time-Equiv}
N(t,s) \leq \frac{1}{-d} (\mu+(c-\lambda)(t-s)) =
 \mu^{'}+(\frac{c}{-d}-\lambda^{'})(t-s) ,\ \forall t\geq s \geq t_0,
\end{equation}
where $\mu^{'}=\frac{\mu}{-d}$ and $\lambda^{'}=\frac{\lambda}{-d}$.

For given $c,d,\lambda,\mu$ denote the set of impulse time sequences, which satisfies \eqref{Dwell-Time-Equiv} by $\Simp_{c,d}[\mu,\lambda]$.

Take $c_1,c_2>0$ and $d_1,d_2<0$ such that $\frac{c_1}{-d_1} >
\frac{c_2}{-d_2}$. Then $\forall \lambda_2,\mu_2>0$ $\exists
\lambda_1,\mu_1>0$: $\Simp_{c_2,d_2}[\mu_2,\lambda_2] \subset
\Simp_{c_1,d_1}[\mu_1,\lambda_1]$. Thus, the set
$\Simp_{c,d}[\mu,\lambda]$ crucially depends on the value of
$\frac{c}{-d}$. We will call $\frac{c}{-d}$ the frequency of impulse
times.

For the gains as in \eqref{GainsDef} the frequency of impulse times is equal to
\begin{equation}
\label{Omega_Def}
\omega(k):=\frac{c_k}{-d}= \frac{\tilde{c}-k}{-d},
\end{equation}
and the possible values of $k$ are contained in $(\rho,\tilde{c})$.
It is clear that $\omega$ is decreasing w.r.t. $k$ on the interval $(\rho,\tilde{c})$, as well as the gains $\Gamma_k$ defined by \eqref{GainsDef}.


%
%
%

We summarize our investigations in the following proposition:
\begin{proposition}
Let $V_i$ be an ISS-Lyapunov function for the $i$-th subsystem, $i=1,\ldots,n$ and the inequalities
\eqref{Dissip_Cont_Inequality} and \eqref{Dissip_Jumps_Inequality} hold with $d<0$ and $\tilde{c} > \rho$ and let the gains are defined as in \eqref{GainsDef}.
Then the possible values of $k$ are contained in $(\rho,\tilde{c})$, and on this interval the smaller are the gains, the smaller is the frequency of impulses allowed by dwell-time condition.
Moreover, $\lim_{k \to \rho}\rho (\Gamma_k) = 1$ and
$\lim_{k \to \tilde{c}}\omega(k) = 0$.

\end{proposition}

\section[Concluding remarks]{Concluding remarks and open questions}
\label{Schluss}

We developed Lyapunov-type stability conditions for impulsive systems for the case when an ISS-Lyapunov function is of general type (nonexponential) as well as when an ISS-Lyapunov function is exponential. To provide the classes of impulse time sequences, for which the system is ISS, we have used nonlinear fixed dwell-time condition from \cite{SaP95} as well as generalized average dwell-time (gADT) condition, which contains ADT condition from \cite{HLT08} as a special case.
The small-gain theorems and linearization method have been generalized to the case of impulsive systems in Sections \ref{Linearisierung_Imp} and \ref{ISS_Kopplungen}. Also we have shown the relation between small-gain and dwell-time condition.

An interesting direction for a future research is a development of the theory of interconnected impulsive systems which subsystems have different sequences of impulse times.

\bibliographystyle{plain}

\bibliography{literatur}

\end{document}